\documentclass[a4paper,12pt,reqno]{amsart}
\usepackage{amsfonts,amssymb,amscd,amsmath,latexsym,amsbsy,enumerate,a4wide}

\newtheorem{thm}{Theorem}[section]
\newtheorem{defin}[thm]{Definition}
\newtheorem{lem}[thm]{Lemma}
\newtheorem{cor}[thm]{Corollary}
\newtheorem{prop}[thm]{Proposition}

\theoremstyle{definition}
\newtheorem{rem}[thm]{Remark}

\numberwithin{equation}{section}

\def\al{\alpha}
\def\be{\beta}
\def\ga{\gamma}
\def\de{\delta}
\def\la{\lambda}
\def\si{\sigma}
\def\Om{\Omega}
\def\La{\Lambda}
\def\Ga{\Gamma}
\def\Q{\mathbb{Q}}
\def\R{\mathbb{R}}
\def\C{\mathbb{C}}
\def\N{\mathbb{N}}
\def\cD{\mathcal D}
\def\cE{\mathcal E}
\def\cR{\mathcal R}
\def\cP{\mathcal P}
\def\sgn{\text{\rm sgn}}
\newcommand{\rFs}[5]{\,_{#1}F_{#2} \left( \genfrac{.}{.}{0pt}{}{#3}{#4};#5 \right)}
\newcommand{\tHe}[3]{\,_{2}H_{1} \left( \genfrac{.}{.}{0pt}{}{#1}{#2};#3 \right)}
\newcommand{\laa}{\mathfrak{a}}
\newcommand{\lak}{\mathfrak{k}}
\newcommand{\lat}{\mathfrak{t}}
\newcommand{\lau}{\mathfrak{u}}
\newcommand{\lasu}{\mathfrak{su}}
\newcommand{\laZ}{\mathfrak{Z}}
\newcommand{\SU}{\mathrm{SU}}
\newcommand{\U}{\mathrm{U}}
\newcommand{\Ad}{\mathrm{Ad}}

\title[Matrix-valued orthogonal polynomials II]
{Matrix-valued orthogonal polynomials related to $(\SU(2)\times \SU(2),\text{diag})$, II}

\author{Erik Koelink, Maarten van Pruijssen, Pablo Rom\'an}
\address{EK, MvP: Radboud Universiteit, IMAPP, 
Heyendaalseweg 135, 
6525 GL Nijmegen, 
The Netherlands}
\email{e.koelink@math.ru.nl m.vanpruijssen@math.ru.nl}
\address{PR: CIEM,
FaMAF, Universidad Nacional de C\'ordoba, Medina Allende s/n Ciudad
Universitaria, C\'ordoba, Argentina}
\email{roman@famaf.unc.edu.ar}
\date{\today }

\begin{document}

\begin{abstract} In a previous paper we have introduced matrix-valued analogues of the Chebyshev 
polynomials by studying matrix-valued spherical functions on $\SU(2)\times \SU(2)$. In particular 
the matrix-size of the polynomials is arbitrarily large. The matrix-valued orthogonal polynomials 
and the corresponding weight function are studied. In particular, we calculate the 
LDU-decomposition of the weight where the matrix entries of $L$ are given in terms of Gegenbauer 
polynomials. The monic matrix-valued orthogonal polynomials $P_n$ are expressed in terms of Tirao's 
matrix-valued hypergeometric function using the matrix-valued differential operator of first and 
second order to which the $P_n$'s are eigenfunctions. From this result we obtain an explicit 
formula for coefficients in the three-term recurrence relation satisfied by the polynomials $P_n$. 
These differential operators are also crucial in expressing the matrix entries of $P_nL$ as a 
product of a Racah and a Gegenbauer polynomial. 
We also present a group theoretic derivation of the matrix-valued differential operators by 
considering the Casimir operators corresponding to  $\SU(2)\times \SU(2)$. 
\end{abstract}

\maketitle
\section{Introduction}\label{sec:intro}

Matrix-valued orthogonal polynomials have been studied from different perspectives in recent years. 
Originally they have been introduced by Krein \cite{Krein1}, \cite{Krein2}. Matrix-valued 
orthogonal polynomials have been related to various different subjects, such as higher-order 
recurrence equations, spectral decompositions, and representation theory. The matrix-valued 
orthogonal polynomials studied in this paper arise from the representation theory of the group 
$\SU(2)\times\SU(2)$ with the compact subgroup $\SU(2)$ embedded diagonally, see \cite{KRvP} for 
this particular case and Gangolli and Varadarajan \cite{VaradarajanGangolli}, Tirao \cite{TiraoSF}, 
Warner \cite{WarnerII} for general group theoretic interpretations of matrix-valued spherical 
functions. An important example is the study of the matrix-valued orthogonal polynomials for the 
case $(\SU(3), \U(2))$, which has been studied by Gr\"unbaum, Pacharoni and Tirao \cite{GPT} mainly 
exploiting the invariant differential operators. In \cite{KRvP} we have studied the matrix-valued 
orthogonal operators related to the case $(\SU(2)\times\SU(2),\SU(2))$, which lead to the 
matrix-valued orthogonal polynomial analogues of Chebyshev polynomials of the second kind $U_n$, in a 
different fashion. In the current paper we study these matrix-valued orthogonal polynomials in more 
detail. 

In order to state the most important results
for these matrix-valued orthogonal polynomials we recall the weight function
\cite[Thm.~5.4]{KRvP}:
\begin{equation}\label{eq:defmatrix_W}
\begin{split}
W(x)_{n,m}\, &=\, \sqrt{1-x^2}\,  \sum_{t=0}^m \al_t(m,n)\, U_{n+m-2t}(x), \\
\al_t(m,n) \, &= \, \frac{(2\ell+1)}{n+1}\frac{(2\ell-m)!m!}{(2\ell)!}
(-1)^{m-t} \frac{(n-2\ell)_{m-t}}{(n+2)_{m-t}}
\frac{(2\ell+2-t)_t}{t!}
\end{split}
\end{equation}
if $n\geq m$ and $W(x)_{n,m}=W(x)_{m,n}$ otherwise. Here and elsewhere in this paper  $\ell\in\frac12\N$, 
$n,m\in \{0,1,\cdots, 2\ell\}$, and $U_n$ is the Chebyshev polynomial of the second kind. Note that the sum in \eqref{eq:defmatrix_W} actually starts at $\min(0, n+m-2\ell)$. 
It follows that $W\colon [-1,1]\to M_{2\ell+1}(\C)$, $W(x)= \bigl( 
W(x)_{n,m}\bigr)_{n,m=0}^{2\ell}$, is a $(2\ell+1)\times (2\ell+1)$-matrix-valued integrable 
function such that all moments 
$\int_{-1}^1 x^n W(x)\, dx$, $n\in\N$, exist. From the construction given in \cite[\S 5]{KRvP}  it 
follows $W(x)$ is positive definite almost everywhere. By general considerations, e.g. 
\cite{GrunT}, we can construct the corresponding monic
matrix-valued orthogonal polynomials $\{P_n\}_{n=0}^\infty$, so
\begin{equation}\label{eq:ortho-monicP}
\langle P_n,P_m\rangle_W = \int_{-1}^1 P_n(x) \, W(x)\, \bigl(P_m(x)\bigr)^\ast\, dx \, = \, \de_{nm} H_n, 
\quad 0< H_n\in M_{2\ell+1}(\C)
\end{equation}
where $H_n>0$ means that $H_n$ is a positive definite matrix, 
 $P_n(x) = \sum_{k=0}^n x^k P^n_{k}$ with $P^n_{k}\in  M_{2\ell+1}(\C)$ and
$P^n_{n}=I$, the identity matrix. The polynomials $P_n$ are the monic variants of the matrix-valued 
orthogonal polynomials constructed in \cite{KRvP} from representation theoretic considerations. 
Note that \eqref{eq:ortho-monicP} defines a matrix-valued inner product $\langle \cdot, 
\cdot\rangle_W$ on the matrix-valued polynomials.  Using the orthogonality relations for the 
Chebyshev polynomials $U_n$  it follows that 
\begin{equation}\label{eq:H0explicitvalue}
(H_0)_{nm} \, = \, \de_{nm} \frac{\pi}{2} \frac{(2\ell+1)^2}{(n+1)(2\ell-n+1)}
\end{equation}
which is in accordance with \cite[Prop.~4.6]{KRvP}. From \cite{KRvP} we can also obtain an expression 
for $H_n$ by translating the result of \cite[Prop.~4.6]{KRvP} to the monic case in 
\cite[(4.6)]{KRvP}, but since the matrix $\Upsilon_d$ in \cite[(4.6)]{KRvP} is relatively 
complicated this leads to a complicated expression for the squared norm matrix  $H_n$ in 
\eqref{eq:ortho-monicP}. In  Corollary \ref{cor:squarednormHn} we give a simpler expression for 
$H_n$ from the three-term recurrence relation. 

These polynomials have a group theoretic interpretation as matrix-valued spherical functions 
associated to $(\SU(2)\times \SU(2), \SU(2))$, see \cite{KRvP} and Section 
\ref{sec:grouptheoreticderivation}.  
In particular, in \cite[\S 5]{KRvP} we have shown that the corresponding orthogonal polynomials are 
not irreducible, but can be written as a $2$-block-diagonal 
matrix of irreducible matrix-valued orthogonal polynomials. Indeed, if we put 
$J\in M_{2\ell}(\C)$, $J_{nm} = \de_{n+m,2\ell}$ we have $JW(x)=W(x)J$ for all $x\in [-1,1]$, and by \cite[Prop.~5.5]{KRvP} $J$ and $I$ span the commutant 
$\{Y\in M_{2\ell}(\C)\mid [Y,W(x)]=0\, \forall\, x\in [-1,1]\}$. Note that $J$ is a self-adjoint involution, $J^2=I$,
$J^\ast=J$. 
It is easier to study the 
polynomials $P_n$, and we discuss the relation to the irreducible cases when appropriate. 

In this paper we continue the study of the matrix-valued orthogonal polynomials and the related 
weight function. Let us discuss in some more detail the results we obtain in this paper. Some of 
these results are obtained employing the group theoretic interpretation and some are obtained using 
special functions. Essentially, we obtain the following results for the weight function:
\begin{enumerate}[(a)]
\item explicit expression for $\text{det}(W(x))$, hence proving \cite[Conjecture 5.8]{KRvP}, see Corollary 
\ref{cor:det-thmLDUdecompW};
\item an LDU-decomposition for $W$ in terms of Gegenbauer polynomials, see Theorem \ref{thm:LDUdecompW}.
\end{enumerate}
Part (a) can be proved by a group theoretic consideration, and gives an alternative proof for a 
related statement by Koornwinder \cite{Koornwinder85}, but we actually calculate it directly from 
(b).   
The LDU-decomposition hinges on expressing the integral of the product of two Gegenbauer 
polynomials and a Chebyshev polynomial as a Racah polynomial, see Lemma \ref{lem:Racah-thmLDUdecompW}. 

For the matrix-valued orthogonal polynomials we obtain the following results:
\begin{enumerate}[(i)]
\item $P_n$ as eigenfunctions to a second-order matrix-valued differential operator $\tilde{D}$ 
and a first-order matrix-valued differential operator $\tilde{E}$, compare \cite[\S 7]{KRvP}, see 
Theorem \ref{thm:differentialoperatorsP} and Section \ref{sec:differentialoperators};
\item the group-theoretic interpretation of $\tilde{D}$ and $\tilde{E}$ using the Casimir operators 
for $\SU(2)\times \SU(2)$, see Section \ref{sec:grouptheoreticderivation}, for which the paper by 
Casselman and Mili{\v{c}}i{\'c} \cite{CM1982} is essential;
\item explicit expressions for the matrix entries of the polynomials $P_n$ in terms of 
matrix-valued hypergeometric series using the matrix-valued differential operators, see Theorem 
\ref{thm:monicRnasMVHF};
\item explicit expressions for the matrix entries of the polynomials $P_nL$ in terms of 
(scalar-valued) Gegenbauer polynomials and Racah polynomials using the LDU-decomposition of the weight $W$ 
and differential operators, see Theorem \ref{thm:cRnasGegenbauertimesRacah};
\item explicit expression for the three-term recurrence satisfied by $P_n$, see Theorem \ref{thm:three_term_for_Rm}.
\end{enumerate}
In particular, (i) and (ii) follow from group theoretic considerations, see Section 
\ref{sec:differentialoperators} and \ref{sec:grouptheoreticderivation}. This then gives the 
opportunity to link the polynomials to the matrix-valued hypergeometric differential operator, 
leading to (iii). 
The explicit expression in (iv) involving Gegenbauer polynomials is obtained by using the 
LDU-decomposition of the weight matrix and the differential operator $\tilde{D}$. The expression of the 
coefficients as Racah polynomials involves the first order differential operator as well. 
Finally, in 
\cite[Thm.~4.8]{KRvP} we have obtained an expression for the coefficients of the three-term 
recurrence relation where the matrix entries of the coefficient matrices are given as sums of 
products of Clebsch-Gordan coefficients, and the purpose of (v) is to give a closed expression for 
these matrices. 
The case $\ell=0$, or the spherical case, corresponds to the Chebyshev polynomials $U_n(x)$, which 
occur as spherical functions for $(\SU(2)\times\SU(2),\SU(2))$ or equivalently as characters on 
$\SU(2)$. For these cases almost all of the statements above reduce to well-known statements for 
Chebyshev polynomials, except that the first order differential has no meaning for this special 
case.

The structure of the paper is as follows. In Section \ref{sec:LDU-weight} we discuss the 
LDU-decomposition of the weight, but the main core of the proof is referred to 
\ref{app:proofthmLDUdecompW}. 
In Section \ref{sec:differentialoperators} we discuss the 
matrix-valued differential operators to which the matrix-valued orthogonal polynomials are eigenfunctions. 
We give a group theoretic proof of this result in Section \ref{sec:grouptheoreticderivation}. In 
\cite[\S 7]{KRvP} we have derived the same operators by a judicious guess and next proving the 
result. In order to connect to Tirao's matrix-valued hypergeometric series, we switch to another 
variable. The connection is made precise in Section \ref{sec:HGfunctions}. This result is next used 
in Section \ref{sec:3termrecurrencerelation} to derive a simple expression for the coefficients in 
the three-term recurrence of the monic orthogonal polynomials, improving a lot on the corresponding 
result \cite[Thm.~4.8]{KRvP}. In Section \ref{sec:MVOPintermsofGegenbauerpols} we explicitly 
establish that the entries of the matrix-valued orthogonal polynomials times the $L$-part of the 
LDU-decomposition of the weight $W$ can be given explicitly as a product of a Racah polynomial and 
a Gegenbauer polynomial, see Theorem \ref{thm:cRnasGegenbauertimesRacah}. Some of the above 
statements require somewhat lengthy and/or tedious manipulations, and in order to deal with these 
computations and also for various other checks we have used computer algebra. 

As mentioned before, we consider the matrix-valued orthogonal polynomials studied in this paper as 
matrix-valued analogues of the Chebyshev polynomials of the second kind. As is well known, the 
group theoretic interpretation of the Chebyshev polynomials, or more generally of spherical 
functions, leads to more information on these special functions, and it remains to study which of 
these properties can be extended in this way to the explicit set of matrix-valued orthogonal 
polynomials studied in this paper. This paper is mainly analytic in nature, and we only use the 
group theoretic interpretation to give a new way on how to obtain the first and second order 
matrix-valued differential operator which have the matrix-valued orthogonal polynomials as 
eigenfunctions. We note that all differential operators act on the right.
The fact that we have both a first and a second order differential operator makes it possible to 
consider linear combinations, and this is useful in Section \ref{sec:HGfunctions} to link to 
Tirao's matrix-valued differential hypergeometric function and Section 
\ref{sec:MVOPintermsofGegenbauerpols} in order to diagonalise (or decouple)  a suitable matrix-
valued differential differential operator. 

We finally remark that J.A.~Tirao has informed us that Ignacio Zurri\'an has obtained results of a 
similar nature by considering matrix-valued orthogonal polynomials for the closely related pair 
$(\mathrm{SO}(4), \mathrm{SO}(3))$. We stress that our results and the results by Zurri\'an have been 
obtained independently. 

\section{LDU-decomposition of the weight}\label{sec:LDU-weight}

In this section we state the LDU-decomposition of the weight matrix $W$ in \eqref{eq:defmatrix_W} 
is discussed. The details of the proof, involving summation and transformation formulas for 
hypergeometric series (up to ${}_7F_6$-level), is presented in 
\ref{app:proofthmLDUdecompW}. 
Some direct consequences of the LDU-decomposition are discussed. The 
explicit decomposition is a crucial ingredient in Section \ref{sec:MVOPintermsofGegenbauerpols}, 
where the matrix-valued orthogonal polynomials are related to the classical Gegenbauer and Racah 
polynomials. 

In order to formulate the result we need the Gegenbauer, or ultraspherical, polynomials, see e.g. 
\cite{AndrAR}, \cite{Isma}, \cite{KoekS}, defined by 
\begin{equation}\label{eq:defGegenbauerpols}
C^{(\al)}_n (x) \, = \, \frac{(2\al)_n}{n!}
\, \rFs{2}{1}{-n, n+2\al}{\al +\frac12}{\frac{1-x}{2}}.
\end{equation}
The Gegenbauer polynomials are orthogonal polynomials; 
\begin{equation}\label{eq:orthorelGegenbauerpols}
\begin{split}
\int_{-1}^1 (1-x^2)^{\al-\frac12} C^{(\al)}_n (x) C^{(\al)}_m (x) dx\, &= \, \de_{nm}
\frac{(2\al)_n\, \sqrt{\pi}\, \Ga(\al+\frac12)}{n!\, (n+\al)\, \Ga(\al)}\, \\ &= \, 
\de_{nm} \frac{\pi\, \Ga(n+2\al)\, 2^{1-2\al}}{\Ga(\al)^2\, (n+\al)\, n!}
\end{split}
\end{equation}

\begin{thm}\label{thm:LDUdecompW}
The weight matrix $W$ has the following LDU-decomposition;
\[
W(x)=\sqrt{1-x^2}\, L(x)\,T(x)\,L(x)^t, \qquad x\in [-1,1], 
\]
where $L\colon [-1,1]\to M_{2\ell+1}(\C)$ is the unipotent lower triangular matrix
\[
L(x)_{mk} = \begin{cases} 0, & k>m \\
\displaystyle{\frac{m!\, (2k+1)!}{(m+k+1)!\, k!} C^{(k+1)}_{m-k}(x)}, & k\leq m
             \end{cases}
\]
and $T\colon [-1,1]\to M_{2\ell+1}(\C)$ is the diagonal matrix
\begin{equation*}
 T(x)_{kk}\, =\, c_k(\ell) 
(1-x^2)^k, \quad c_k(\ell)\, = \, 
\frac{4^k (k!)^4 (2k+1)}{((2k+1)!)^2} \frac{(2\ell+k+1)!\, (2\ell-k)!}{((2\ell)!)^2}.
\end{equation*}
\end{thm}

Note that the matrix-entries of $L$ are independent of $\ell$, hence of the size of the matrix-valued weight $W$. Using
\[
\frac{d^k}{dx^k} C^{(\al)}_n(x) = 2^k (\al)_k \, C^{(\al+k)}_{n-k}(x)
\]
we can write uniformly $L(x)_{mk} = \frac{m!\, 2^{-k}\, (2k+1)!}{(k!)^2\, (m+k+1)!}\, \frac{d^kU_m}{dx^k}(x)$. 
In Theorem \ref{thm:cRnasGegenbauertimesRacah} we extend Theorem \ref{thm:LDUdecompW}, but Theorem \ref{thm:LDUdecompW} is an essential ingredient in Theorem \ref{thm:cRnasGegenbauertimesRacah}. 

Since $W(x)$ is symmetric, it suffices to consider $(n,m)$-matrix-entry for $m\leq n$ of  Theorem \ref{thm:LDUdecompW}. Hence Theorem \ref{thm:LDUdecompW} follows directly from  
Proposition \ref{prop:thmLDUdecompW} using the explicit expression \eqref{eq:defmatrix_W} for the weight $W$. 

\begin{prop}\label{prop:thmLDUdecompW} The following relation 
 \begin{equation*}
\begin{split}
\sum_{t=0}^m \al_t(m,n) U_{n+m-2t}(x) = \sum_{k=0}^m  
\be_k(m,n) (1-x^2)^k \, C^{(k+1)}_{n-k}(x)C^{(k+1)}_{m-k}(x)
\end{split}
\end{equation*}
with the coefficients $\al_t(m,n)$ given by \eqref{eq:defmatrix_W} and 
\begin{equation*}
\begin{split}
\be_k(m,n) = \frac{m!}{(m+k+1)!} \frac{n!}{(n+k+1)!}
k!\, k!\, 2^{2k} (2k+1) \frac{(2\ell+k+1)!\ (2\ell-k)!}{(2\ell)!\, (2\ell)!} 
\end{split} 
\end{equation*}
holds for all integers $0\leq m\leq n\leq 2\ell$, and all $\ell\in\frac12\N$.  
\end{prop}

Before discussing the proof we list some corollaries of 
Theorem \ref{thm:LDUdecompW}. First of all, we can use Theorem \ref{thm:LDUdecompW} to prove \cite[Conjecture 5.8]{KRvP}, see (a) of Section \ref{sec:intro}.

\begin{cor}\label{cor:det-thmLDUdecompW}
$\det\bigl(W(x)\bigr) =  (1-x^2)^{2(\ell+\frac12)^2} 
\prod_{k=0}^{2\ell} c_k(\ell)$. 
\end{cor}

\begin{rem}\label{rmk:cor-UDL-thmLDUdecompW} We also have another proof of this fact using a 
group theoretic approach to calculate $\det(\Phi_0(x))$, see \cite{KRvP} and Section 
\ref{sec:grouptheoreticderivation} for the definition of $\Phi_0$, and  $W$ is up to trivial 
factors equal to $(\Phi_0)(\Phi_0)^\ast$. This proof is along the lines of Koornwinder 
\cite{Koornwinder85}.  
\end{rem}

Secondly, using $J\in M_{2\ell+1}(\C)$, $J_{nm}=\de_{n+m,2\ell}$ and $W(x)=JW(x)J$, see 
\cite[Prop.~5.5, \S 6.2]{KRvP}, we obtain from Theorem \ref{thm:LDUdecompW} the UDL-decomposition 
for $W$. For later reference we also recall $JP_n(x)J=P_n(x)$, since both are the monic matrix-
valued orthogonal polynomials with respect to $W(x)=JW(x)J$.

\begin{cor}\label{cor:UDL-thmLDUdecompW}
$W(x) = \sqrt{1-x^2} \bigl(JL(x)J\bigr)\bigl(JT(x)J\bigr) \bigl(JL(x)J\bigr)^t$, $x\in [-1,1]$ 
gives the UDL-decomposition of the weight $W$. 
\end{cor}

Thirdly, considering the Fourier expansion of the weight function $W(\cos\theta)$, and using the 
expression of the weight in terms of Clebsch-Gordan coefficients, see \cite[(5.4), (5.6), (5.7)]
{KRvP} we obtain a Fourier expansion, which is actually equivalent to Theorem \ref{thm:LDUdecompW}.

\begin{cor}\label{cor:Fourier-thmLDUdecompW} We have the following Fourier expansion
\begin{equation*}
\begin{split}
&\sum_{k=0}^{m\wedge n} (-4)^k (2k+1) \frac{(m-k+1)_k\, (n-k+1)_k}{(m+1)_{k+1}\, (n+1)_{k+1}} 
\frac{(2\ell+k+1)!\ (2\ell-k)!}{(2\ell)!\, (2\ell)!} 
e^{-i(n+m)t}  \\ &\qquad 
\times (1-e^{2it})^{2k} \rFs{2}{1}{k-n,k+1}{-n}{e^{2it}}
\rFs{2}{1}{k-m,k+1}{-m}{e^{2it}} \, = \,  \\
&\sum_{j=0}^{2\ell} \sum_{j_1=0}^n \sum_{\substack{j_2=0 \\ j_1+j_2=j}}^{2\ell-n}
\sum_{i_1=0}^m\ \sum_{\substack{i_2=0 \\ i_1+i_2=j}}^{2\ell-m}
\frac{\binom{n}{j_1}\binom{2\ell-n}{j_2}}{\binom{2\ell}{j}} 
\frac{\binom{m}{i_1}\binom{2\ell-m}{i_2}}{\binom{2\ell}{j}} 
e^{i( (n-j_1+j_2)-(m-i_1+i_2))t}
\end{split}
\end{equation*}
\end{cor}

\begin{proof} In \cite[\S 5, 6]{KRvP} the weight function $W(\cos t)$ was initially defined as a 
Fourier polynomial with the coefficients given in terms of Clebsch-Gordan coefficients. After 
relabeling this gives
\begin{equation*}
\begin{split}
& \sum_{j=0}^{2\ell} \sum_{j_1=0}^n \sum_{\substack{j_2=0 \\ j_1+j_2=j}}^{2\ell-n}
\sum_{i_1=0}^m\ \sum_{\substack{i_2=0 \\ i_1+i_2=j}}^{2\ell-m}
\frac{\binom{n}{j_1}\binom{2\ell-n}{j_2}}{\binom{2\ell}{j}} 
\frac{\binom{m}{i_1}\binom{2\ell-m}{i_2}}{\binom{2\ell}{j}} 
e^{i( (n-j_1+j_2)-(m-i_1+i_2))t} = \\
&\bigl( L(\cos t) T(\cos t) L(\cos t)^t\bigr)_{nm} =
\hskip-.3truecm \sum_{k=0}^{\min(m,n)}  
\be_k(m,n) \sin^{2k} t \, C^{(k+1)}_{n-k}(\cos t)C^{(k+1)}_{m-k}(\cos t)
\end{split}
\end{equation*}
where we have used \cite[(5.10)]{KRvP} to express the Clebsch-Gordan coefficients in terms of 
binomial coefficients. 

Using the result \cite[Cor. 6.3]{BadeK} by Koornwinder and Badertscher together with the Fourier 
expansion of the Gegenbauer polynomial, see \cite[(2.8)]{BadeK}, \cite[(6.4.11)]{AndrAR}, \cite[(4.5.13)]{Isma}, we find the Fourier expansion of 
$\sin^k t\, C^{(k+\la)}_{n-k}(\cos t)$ in terms of Hahn polynomials defined by
\begin{equation}
Q_k(j;\al,\be,N) = \rFs{3}{2}{-k,k+\al+\be+1,-j}{\al+1, -N}{1}, \qquad
k\in \{0,1,\cdots, N\},
\end{equation}
see \cite[p.~345]{AndrAR}, \cite[\S 6.2]{Isma}, \cite[\S 1.5]{KoekS}. 
For $\la=1$  the explicit formula is  
\begin{equation}\label{eq:BKexpansion}
\begin{split}
&\frac{i^k (n+1)_{k+1}\, (n-k)!}{2^k (\frac32)_k\, (2k+2)_{n-k}}
\sin^k t\, C^{(k+1)}_{n-k}(\cos t) = 
\sum_{j=0}^n  
Q_k(j;0,0,n) e^{i(2j-n)t}\\ &\qquad\qquad\qquad \qquad\qquad\, =\, e^{-int} (1-e^{2it})^k \rFs{2}{1}{k-n, k+1}{-n}{e^{2it}}
\end{split}
\end{equation}
using the generating function \cite[(1.6.12)]{KoekS} for the Hahn polynomials in the last equality. Plugging this in the identity gives the required result. 
\end{proof}

In the proof of Proposition \ref{prop:thmLDUdecompW} and Theorem \ref{thm:LDUdecompW} given in 
\ref{app:proofthmLDUdecompW} 
we use a somewhat unusual integral representation of a Racah polynomial. Recall the Racah polynomials,
\cite[p.~344]{AndrAR}, \cite[\S 1.2]{KoekS}, defined by 
\begin{equation}\label{eq:defRacahpols}
R_k(\la(t);\al,\be,\ga,\de)\, =\, \rFs{4}{3}{-k, k+\al+\be+1, -t, t+\ga+\de+1}{\al+1, \be+\de+1,\ga+1}{1}
\end{equation}
where $\la(t)=t(t+\ga+\de+1)$, and one out of $\al+1$, $\be+\de+1$, $\ga+1$ equals $-N$ with a non-negative integer $N$. 
The Racah polynomials with $0\leq k\leq N$ form a set of orthogonal polynomials for $t\in \{0,1,\cdots, N\}$ for suitable conditions on the parameters. 
For the special case of the Racah polynomials in Lemma \ref{lem:Racah-thmLDUdecompW} the orthogonality relations are given in 
\ref{app:proofthmLDUdecompW}. 

\begin{lem}\label{lem:Racah-thmLDUdecompW} 
For integers $0\leq t, k\leq m\leq n$ we have 
\begin{gather*}
\int_{-1}^1 (1-x^2)^{k+\frac12} C^{(k+1)}_{n-k}(x)C^{(k+1)}_{m-k}(x) U_{n+m-2t}(x) 
\, dx 
= 
\frac{\sqrt{\pi}\, \Ga(k+\frac32)}{(k+1)}
\frac{(k+1)_{m-k}}{(m-k)!} \\ \times  \frac{(k+1)_{n-k}}{(n-k)!} 
\frac{(-1)^{k}\, (2k+2)_{m+n-2k}\, (k+1)!}{(n+m+1)!} 
R_k(\la(t);0,0,-n-1,-m-1)
\end{gather*}
\end{lem}

\begin{rem}\label{rmk:lem:Racah-thmLDUdecompW} Lemma \ref{lem:Racah-thmLDUdecompW} can be extended using the same method of proof to  
\begin{equation}\label{eq:result2a}
\begin{split}
&\int_{-1}^1 (1-x^2)^{\al+k+\frac12} C^{(\al+k+1)}_{n-k}(x)C^{(\al+k+1)}_{m-k}(x) C^{(\be)}_{n+m-2t}(x) 
\, dx  = \\ 
&\frac{(\al+k+1)_{m-k}\, (2k+2\al+2)_{n-k}}{(m-k)!\, (n-m)!} \frac{(-m+\be-\al-1)_{m-t}}{(m-t)!} 
\frac{(\be)_{n-t}\sqrt{\pi}\, \Ga(\al+k+\frac32)}{\Ga(\al+n+m-t+2)} \\
&\times \rFs{4}{3}{k-m, -m-2\al-k-1,t-m,\be+n-t}{\be-\al-1-m,-m-\al, n-m+1}{1}
\end{split}
\end{equation}
assuming $ n\geq m$.
Lemma \ref{lem:Racah-thmLDUdecompW} corresponds to the case $\al=0$, $\be=1$ after using a 
transformation for a balanced ${}_4F_3$-series. Note that ${}_4F_3$-series can be expressed as a 
Racah polynomial orthogonal on $\{0,1,\cdots, m\}$ in case $\al=0$ or $\be=\al+1$, which 
corresponds to Lemma \ref{lem:Racah-thmLDUdecompW}. We do not use \eqref{eq:result2a} in the paper, 
and a proof follows the lines of the proof of Lemma \ref{lem:Racah-thmLDUdecompW} as given in 
\ref{app:proofthmLDUdecompW}. 
\end{rem}

In \cite[Thm.~6.5]{KRvP}, see Section \ref{sec:intro}, we have proved that the weight function $W$ 
is not irreducible, meaning that there exists $Y\in M_{2\ell+1}(\C)$ so that 
\begin{equation}\label{eq:irreduciblesplitW(x)}
 YW(x)Y^t\, = \, \begin{pmatrix} W_1(x) & 0 \\ 0 & W_2(x) \end{pmatrix}, \qquad
YY^t \, = \, I \, = \, Y^tY
\end{equation}
and that there is no further reduction.

Writing 
\[
Y \, = \, \begin{pmatrix} A & B \\ C & D \end{pmatrix}, 
\quad L(x) \, = \, \begin{pmatrix} l_1(x) & 0 \\ r(x) & l_2(x) \end{pmatrix}, 
\quad T(x) \, = \, \begin{pmatrix} t_1(x) & 0 \\ 0 & t_2(x) \end{pmatrix}, 
\]
with $A$, $D$ diagonal and $B$, $C$ antidiagonal, see \cite[Cor.~5.6]{KRvP}, $l_1(x)$, $l_2(x)$ 
lower-diagonal matrices and $r(x)$ a full matrix, we can work out the block-diagonal structure of 
$YL(x)T(x)L(x)^tY^t$. It follows that the off-diagonal blocks being zero is equivalent to 
\begin{equation}\label{eq:resultforGegenbauerpols}
\left( Al_1(x)t_1(x) + Br(x)t_1(x)\right)\left(l_1(x)^tC^t+r(x)^tD^t\right) + 
Bl_2(x)t_2(x)l_2(x)^tD^t = 0.
\end{equation}
This can be rewritten as an identity for four sums of products of two Gegenbauer polynomials 
involving the weight function and the constants in Theorem \ref{thm:LDUdecompW} being zero. We do 
not write the explicit results, since we do not need them.

\section{Matrix-valued orthogonal polynomials as eigenfunctions of matrix-valued differential operators}\label{sec:differentialoperators}

In \cite[\S 7]{KRvP} we have derived that the matrix-valued orthogonal polynomials are 
eigenfunctions for a second and a first order matrix-valued differential operator by looking for 
suitable  matrix-valued differential operators self-adjoint with respect to the matrix-valued inner 
product $\langle\cdot,\cdot\rangle_W$. The method was to establish relations between the 
coefficients of the differential operators and the weight $W$, next judiciously guessing the 
general result and next proving it by a verification. 
In this paper we show that essentially these operators can be obtained from the group theoretic 
interpretation by establishing that the matrix-valued differential operators are obtainable from 
the Casimir operators for $\SU(2)\times\SU(2)$. Since the paper is split into a first part of 
analytic nature and a second part of group theoretic nature, we state the result in this section 
whereas the proofs are given in Section \ref{sec:grouptheoreticderivation}. Sections 
\ref{sec:HGfunctions} and \ref{sec:MVOPintermsofGegenbauerpols} depend strongly on the 
matrix-valued differential operators in Theorem \ref{thm:differentialoperatorsP}. 

Recall that all differential operators act on the right, so for  a matrix-valued polynomial 
$P\colon \R \to M_N(\C)$ depending on the variable $x$, the 
$s$-th order differential operator $D = \sum_{i=0}^s \frac{d^i}{dx^i} F_i(x)$, $F_i\colon \R \to 
M_N(\C)$,  acts by
\[
\bigl(PD\bigr)(x)\, =\,  \sum_{i=0}^s \frac{d^iP}{dx^i}(x) F_i(x), \qquad PD \colon \R \to M_N(\C)
\]
where 
$\bigl( \frac{d^iP}{dx^i}(x)\bigr)_{nm} = \frac{d^iP_{nm}}{dx^i}(x)$ is a matrix which is 
multiplied from the right by the matrix $F_i(x)$.  The matrix-valued orthogonal polynomial is an 
eigenfunction of a matrix-valued differential operator if there exists a matrix $\La\in M_N(\C)$, 
the eigenvalue matrix, so that $PD=\La P$ as matrix-valued functions. Note that the eigenvalue 
matrix is multiplied from the left. 
For more information on differential operators for matrix-valued functions, see e.g. \cite{GrunT}, 
\cite{TiraPNAS}. 

We denote by $E_{ij}$ the standard matrix units, i.e. $E_{ij}$ is the matrix with all matrix entries equal to zero, except for the $(i,j)$-th entry which is $1$. By convention, if either $i$ or $j$ is not in the appropriate range, the matrix $E_{ij}$ is zero. 

\begin{thm}\label{thm:differentialoperatorsP} Define the second order matrix-valued differential operator
\begin{gather*}
\tilde{D}\, =\, (1-x^2)\, \frac{d^2}{dx^2}\, + \, \left( \frac{d}{dx}\right) (\tilde{C}-x\tilde{U}) \, - \, \tilde{V}\\
\tilde{C} = \sum_{i=0}^{2\ell} (2\ell-i) E_{i,i+1} + \sum_{i=0}^{2\ell} i E_{i,i-1}, \quad
\tilde{U} = (2\ell+3)I, \quad \tilde{V} = -\sum_{i=0}^{2\ell} i(2\ell-i) E_{ii}
\end{gather*}
and the first order matrix-valued differential operators
\begin{gather*}
\tilde{E}\, =\, \left( \frac{d}{dx}\right) (\tilde{B}_0+x\tilde{B}_1) \, + \, \tilde{A} \\
\tilde{B}_0=-\sum_{i=0}^{2\ell} \frac{(2\ell-i)}{4\ell} E_{i,i+1} 
+\sum_{i=0}^{2\ell} \frac{(\ell-i)}{2\ell}E_{ii} + \sum_{i=0}^{2\ell} \frac{i}{4\ell}E_{i,i-1},\\ 
\tilde{B}_1=-\sum_{i=0}^{2\ell} \frac{(\ell-i)}{\ell} E_{i,i}, \qquad
\tilde{A}=\sum_{i=0}^{2\ell} \frac{(2\ell+2)(i-2\ell)}{-4\ell} E_{i,i}, \
\end{gather*}                           
then the monic orthogonal matrix-valued orthogonal polynomials $P_n$ satisfy
\begin{gather*}
P_n\tilde{D}\, = \, \La_n(\tilde{D})P_n, \qquad \La_n(\tilde{D}) =  \sum_{i=0}^{2\ell} \left(-n(n-1)-n(2\ell+3)+i(2\ell-i)\right)E_{ii}, \\ 
P_n\tilde{E}\, = \, \La_n(\tilde{E})P_n, \qquad \La_n(\tilde{E}) = 
\sum_{i=0}^{2\ell} \left(\frac{n(\ell-i)}{2\ell}-\frac{(2\ell+2)(i-2\ell)}{4\ell}\right)E_{ii}. \\ 
\end{gather*}
and the operators $\tilde{D}$ and $\tilde{E}$ commute. The operators are symmetric with respect to $W$.
\end{thm}

The group theoretic proof of Theorem \ref{thm:differentialoperatorsP} is given in Section \ref{sec:grouptheoreticderivation}. Theorem \ref{thm:differentialoperatorsP} has been proved in 
\cite[Thms.~7.5, 7.6]{KRvP} analytically.  
The symmetry of the operators with respect to $W$ means that $\langle PD,Q\rangle_W=\langle 
P,QD\rangle_W$ and $\langle PE,Q\rangle_W=\langle P,QE\rangle_W$ for all matrix-valued polynomials 
with respect to the matrix-valued inner product $\langle\cdot, \cdot\rangle_W$ defined in 
\eqref{eq:ortho-monicP}. The last statement follows immediately from the first by the results of 
Gr\"unbaum and Tirao \cite{GrunT}. Also, $[D,E]=0$ follows from the fact that the eigenvalue 
matrices commute. In the notation of \cite{GrunT} we have
$\tilde{D}, \tilde{E}\in \cD(W)$, where $\cD(W)$ is the $\ast$-algebra of matrix-valued 
differential operators having the matrix-valued orthogonal polynomials as eigenfunctions. 

Note that $\tilde{E}$ has no analogue in case $\ell=0$, whereas $\tilde{D}$ reduces to the 
hypergeometric differential operator for the Chebyshev polynomials $U_n$. 

The matrix-differential operator $\tilde{D}$ is $J$-invariant, i.e. $J\tilde{D}J = \tilde{D}$. The 
operator $\tilde{E}$ is almost $J$-anti-invariant, up to a multiple of the identity. This is 
explained in 
Theorem \ref{thm:DOfromgrouptoDE} and the discussion following this theorem. 
In particular, $\tilde{D}$ descends to the corresponding irreducible matrix-valued orthogonal polynomials, but $\tilde{E}$ does not, see also \cite[\S 7]{KRvP}. 

\section{Matrix-valued orthogonal polynomials as matrix-valued hypergeometric functions}\label{sec:HGfunctions}

The polynomial solutions to the hypergeometric differential equation, see 
\eqref{eq:HGdiffopscalar}, are uniquely determined. Many classical orthogonal polynomials, such as 
the Jacobi, Hermite, Laguerre and Chebyshev, can be written in terms of hypergeometric series. For 
matrix-valued valued functions Tirao \cite{TiraPNAS} has introduced a matrix-valued hypergeometric 
differential operator and its solutions. The purpose of this section is to link the monic 
matrix-valued orthogonal polynomials to Tirao's matrix-valued hypergeometric functions. 

We want to use Theorem \ref{thm:differentialoperatorsP} in order to express the matrix-valued 
orthogonal polynomials as matrix-valued hypergeometric functions using Tirao's approach 
\cite{TiraPNAS}. In order to do so we have to switch from the interval $[-1,1]$ to $[0,1]$ using 
$x=1-2u$. We define 
\begin{equation}\label{eq:defpolsRn}
R_n(u)\, = \,  (-1)^n 2^{-n} P_n(1-2u), \qquad Z(u) \, = \, W(1-2u)
\end{equation}
so that the rescaled monic matrix-valued orthogonal polynomials $R_n$ satisfy 
\begin{equation}
Z(u) = W(1-2u), \qquad \int_0^1 R_n(u) Z(u) R_m(u)^\ast\, du \, = \,  2^{-1-2n} H_n. 
\end{equation}
In the remainder of Section \ref{sec:HGfunctions} we work with the polynomials $R_n$ on the interval $[0,1]$. 
It is a straightforward check to rewrite Theorem \ref{thm:differentialoperatorsP}. 

\begin{cor}\label{cor:thmdifferentialoperatorsP}
Let $D$ and $E$ be the matrix valued differential operators
\[
D\,=\,u(1-u)\frac{d^2}{du^2}+\left(\frac{d}{du}\right)(C-uU)-V,\qquad 
E\,=\,\left(\frac{d}{du}\right)(uB_1+B_0)+A_0,
\]
where the matrices $C$, $U$, $V$, $B_0$, $B_1$ and $A_0$ are given by
\begin{gather*}
C=-\sum_{i=0}^{2\ell} \frac{(2\ell-i)}{2} E_{i,i+1} 
+\sum_{i=0}^{2\ell} \frac{(2\ell+3)}{2}E_{ii} - \sum_{i=0}^{2\ell} \frac{i}{2}E_{i,i-1},\quad U=(2\ell+3)I,\\
V= -\sum_{i=0}^{2\ell} i(2\ell-i) E_{i,i}, \quad A_0=\sum_{i=0}^{2\ell} \frac{(2\ell+2)(i-2\ell)}{2\ell} E_{i,i}, \quad 
B_1=-\sum_{i=0}^{2\ell} \frac{(\ell-i)}{\ell} E_{i,i}, \\
B_0=-\sum_{i=0}^{2\ell} \frac{(2\ell-i)}{4\ell} E_{i,i+1} 
+\sum_{i=0}^{2\ell} \frac{(\ell-i)}{2\ell}E_{ii} + \sum_{i=0}^{2\ell} \frac{i}{4\ell}E_{i,i-1}. 
\end{gather*}
Then $D$ and $E$ are symmetric with respect to the weight $W$, and $D$ and $E$ commute. Moreover for every integer $n\geq 0$,
\begin{gather*}
R_nD=\La_n(D)R_n,\qquad \La_n(D)= \sum_{i=0}^{2\ell} \left(-n(n-1)-n(2\ell+3)+i(2\ell-i)\right)E_{ii},\\
R_nE=\Lambda_n(E)R_n, \qquad \La_n(E)= \sum_{i=0}^{2\ell} \left(-\frac{n(\ell-i)}{\ell}+\frac{(2\ell+2)(i-2\ell)}{2\ell}\right)E_{ii}. 
\end{gather*}
\end{cor}

It turns out that it is more convenient to work with $D_\al=D+\al E$ for $\al\in\R$, so that 
$R_nD_\al = \La_n(D_\al)R_n$ with diagonal eigenvalue matrix $\La_n(D_\al)=\La_n(D) + \al \La_n(E)$. 
By \cite[Prop.~2.6]{GrunT}  we have $\La_n(D_\al)= -n^2-n(U_\al-1)-V_\al$. 
Since the eigenvalue matrix $\La_n(D_\al)$ is diagonal, the matrix-valued differential equation 
$R_nD_\al \, = \, \La_n(D_\al)R_n$ can be read as $2\ell+1$ differential equations for the rows of 
$R_n$. The $i$-th row of $R_n$ is a solution to 
\begin{equation}\label{eq:hypergeometric_with_lambda}
u(1-u)p''(u)+p'(u)(C_\al-uU_\al)-p(u)(V_\al+\la)=0, \qquad \la=\bigl(\La_n(D_\al)\bigr)_{ii}
\end{equation}
for $p\colon \C \to \C^{2\ell+1}$ a (row-)vector-valued polynomial function. Here 
$C_\alpha=C+\alpha B_0$, $U_\alpha=U-\alpha B_1$, $V_\alpha=V-\alpha A_0$ using the notation of 
Corollary \ref{cor:thmdifferentialoperatorsP}. Now \eqref{eq:hypergeometric_with_lambda} allows us 
to connect to Tirao's matrix-valued hypergeometric function \cite{TiraPNAS}, which we briefly 
recall in Remark \ref{rmk:TiraoMVHF}. 

\begin{rem}\label{rmk:TiraoMVHF} 
Given $d\times d$ matrices $C$, $U$ and $V$ we can consider the differential equation 
\begin{equation}\label{eq:Tirao-hypergeometric_equation}
z(1-z)F''(z)\,+\,(C-zU)F'(z)\, -\, VF(z)=0, \quad z\in\C,
\end{equation}
where $F\colon \C\to \mathbb{C}^d$ is a (column-)vector-valued function which is twice differentiable.
It is shown by Tirao \cite{TiraPNAS} that if the eigenvalues of $C$ are not in $-\N$, then  
the matrix valued hypergeometric function ${}_2H_1$ defined as the power series
\begin{equation}\label{eq:def2H1-Tirao}
\begin{split}
& \qquad\qquad \tHe{U, V}{C}{z} \, =\, \sum_{i=0}^\infty
\frac{z^i}{i!} \, [C,U,V]_i,  \\
&[C,U,V]_0=1, \quad [C,U,V]_{i+1}=(C+i)^{-1}\bigl( i^2+i(U-1)+V\bigr) [C,U,V]_i
\end{split}
\end{equation}
converges for $|z|<1$ in $M_d(\C)$. Moreover, for $F_0\in \C^d$ the (column-)vector-valued function
\[
F(z) \, = \, \tHe{U, V}{C}{z}\, F_0
\]
is a solution to \eqref{eq:Tirao-hypergeometric_equation} which is analytic for $|z|<1$, and any analytic (on $|z|<1$)  solution to \eqref{eq:Tirao-hypergeometric_equation} is of this form. 
\end{rem}

Comparing Tirao's matrix-valued hypergeometric differential equation 
\eqref{eq:Tirao-hypergeometric_equation} with \eqref{eq:hypergeometric_with_lambda} and using Remark 
\ref{rmk:TiraoMVHF}, we see that 
\begin{equation}\label{eq:explicitsolintermsof2H1}
p(u) \, = \, \left( \tHe{U_\al^t, V_\al^t+\la}{C_\al^t}{u}\, P_0\right)^t \, = \, 
P_0^t \left( \tHe{U_\al^t, V_\al^t+\la}{C_\al}{u}\right)^t, 
\end{equation}
$P_0\in\C^{2\ell+1}$, are the solutions to \eqref{eq:hypergeometric_with_lambda} which are analytic in $|u|<1$ assuming 
that eigenvalues of $C_\al^t$ are not in $-\N$. We first verify this assumption. Even though 
$V_\al$ and $U_\al$ are symmetric, we keep the notation for transposed matrices for notational 
esthetics. 

\begin{lem}\label{lem:eigenvalues_C}
For every $\ell\in\frac12\N$, the matrix $C_\al$ is a diagonalizable
matrix with eigenvalues $(2j+3)/2$, $j\in\{0,\ldots,2\ell\}$. 
\end{lem}

\begin{proof} Note that $C_\al$ is tridiagonal, so that $v_\la=\sum_{n=0}^{2\ell} p_n(\la)e_n$  
is an eigenvector for $C_\al$ for the eigenvalue $\la$ if and only if
\begin{multline*}
-(\la-\frac32)\, p_n(\la) \, = \,\frac{(2\ell-n)(2\ell+\al)}{4\ell}\, p_{n+1}(\la) 
\\ -\left( \frac{(2\ell+\al)(2\ell-n)+n(2\ell-\al)}{4\ell}\right)\, p_n(\la) 
+\frac{(2\ell-\al)n}{4\ell} \, p_{n-1}(\la).
\end{multline*}
The three-term recurrence relation corresponds precisely to the three-term recurrence relation for the Krawtchouk polynomials for $N\in\N$, 
\[
K_n(x;p,N)\, = \,\rFs{2}{1}{-n, -x}{-N}{\frac{1}{p}}, \qquad n,x\in \{0,1,\cdots, N\},
\]
see e.g. \cite[p.~347]{AndrAR}, \cite[\S 6.2]{Isma}, \cite[\S 1.10]{KoekS}, with $N=2\ell$, $p=\frac{2\ell+\al}{4\ell}$. 
The Krawtchouk polynomials are orthogonal with respect to the binomial distribution for $0<p<1$, or
 $\al\in (-2\ell, 2\ell)$, and we find 
\[
p_n(\la) \, = \, K_n(\la; \frac{2\ell+\al}{4\ell}, \ell)
\]
and the eigenvalues of $C_\al$ are $\frac32+x$, $x\in \{0,1,\cdots, 2\ell\}$. This proves the statement for 
$\al\in (-2\ell, 2\ell)$. 

Note that for $\al\not=\pm 2\ell$, the matrix $C_\al$ is tridiagonal, and the eigenvalue equation is solved by the same contiguous relation for the ${}_2F_1$-series leading to the same statement for $|\al|>2\ell$. 
In case $\al=\pm2\ell$ the matrix $C_{\pm 2\ell}$ is upper or lower triangular, and the eigenvalues can be read off from the diagonal.
\end{proof}

In particular, we can give the eigenvectors of $C_\al$ explicitly in terms of terminating ${}_2F_1$-hypergeometric series, but we do not use the result in the paper.

So \eqref{eq:explicitsolintermsof2H1} is valid and this gives a series representation for the rows 
of the monic polynomial $R_n$. Since each row is polynomial, the series has to terminate. This 
implies that 
there exists $n\in\N$ so that $[C_\al^t,U_\al^t,V_\al^t+\la]_{n+1}$ is singular and 
$0\not= P_0\in \text{Ker}\left([C_\al^t,U_\al^t,V_\al^t+\la]_{n+1}\right)$. 

Suppose that 
$n$ is the least integer for which $[C_\al^t,U_\al^t,V_\al^t+\la]_{n+1}$ is singular, i.e.
$[C_\al^t,U_\al^t,V_\al^t+\la]_{i}$ is regular for all $i\leq n$. Since
\begin{equation}\label{eq:MV-Pochhammer}
[C_\al^t,U_\al^t,V_\al^t+\la]_{n+1}= (C_\al^t+n)^{-1}\left(n^2+n(U_\al^t-1)+V_\al^t+\la\right)[C_\al^t,U_\al^t,V_\al^t+\la]_{n}
\end{equation}
and since the matrix $(C_\al+n)$ is invertible by Lemma \ref{lem:eigenvalues_C},  
$[C_\al^t,U_\al^t,V_\al^t+\la]_{n+1}$ is a singular matrix if and only if the diagonal matrix
\begin{equation}\label{eq:defMnla}
\begin{split}
M_n^\al(\la)\, &=\,\left(n^2+n(U_\al^t-1)+V_\al^t+\la\right)\,
\\ &=\,\left(n^2+n(U_\al-1)+V_\al+\la\right) \, = \,
\la-\La_n(D_\al)
\end{split}
\end{equation}
is singular. Note that the diagonal entries of $M_n^\al(\la)$ are of the form $\la-\la_j^\al(n)$, 
so that $M_n(\la)$ is singular if and only if $\la = \la_j^\al(n)$ for some $j\in\{0,1,\cdots, 
2\ell\}$. We need that the eigenvalues are sufficiently generic. 

\begin{lem}\label{lem:eigenvalues_lambda}
Let $\al\in\R\setminus\Q$. Then $(j,n)= (i,m)\in \{0,1\cdots,2\ell\}\times\N$ if and only if 
$\la^\al_j(n) = \la^\al_{i}(m)$.
\end{lem}

\begin{proof} Assume $\la^\al_j(n)=\la^\al_{i}(m)$ and 
let $(j,n), (i,m)\in \{0,1\cdots,2\ell\}\times\N$, then 
\begin{equation*} 
0 \, =\,  \la^\al_j(n)-\la^\al_{i}(m)\, =\, 
(m-n)(n+m+2+2\ell)+(j-i)\left( \frac{\alpha(2\ell+4)}{2\ell}-j-i+2\ell\right). 
\end{equation*}
If $j\neq i$, then we solve for $\al=\frac{2\ell}{(2\ell+4)} \left( -\frac{(m-n)(m+n+2\ell+2)}{j-i}+j+i-2\ell \right)$ which is rational. 

Assume next that $j=i$, then $(m-n)(n+m+2+2\ell)=0$. 
Since $n,m,\ell\geq0$, it follows that $n=m$ and hence $(j,n)= (i,m)$. 
\end{proof}

Assume $\al$ irrational, so that Lemma \ref{lem:eigenvalues_lambda} 
shows that $M_n\bigl(\la^\al_i(m)\bigr)$ is singular if and only if $n=m$.  
So in the series \eqref{eq:explicitsolintermsof2H1} the matrix 
$[C_\al^t,U_\al^t,V_\al^t+\la]_{n+1}$ is singular 
and $[C_\al^t,U_\al^t,V_\al^t+\la]_{i}$ is non-singular for $0\leq i\leq n$. Furthermore, by Lemma 
\ref{lem:eigenvalues_lambda} we see that the kernel of $[C_\al^t,U_\al^t,V_\al^t+\la]_{n+1}$ is 
one-dimensional if and only if $\la=\la^\al_n(i)$, $i\in\{0,1,\cdots,2\ell\}$. In case 
$\la=\la^\al_n(i)$ 
\[
P_0 \, = \, [C_\al^t, U_\al^t, V_\al^t+\la^\al_n(i)]_n^{-1} e_i
\]
is determined uniquely up to a scalar, where $e_i$ is the standard basis vector. 

We can now state the main result of this section, expressing the monic polynomials $R_n$ as a 
matrix-valued hypergeometric function. 

\begin{thm}\label{thm:monicRnasMVHF} 
With the notation of Remark \ref{rmk:TiraoMVHF} the monic matrix-valued orthogonal polynomials can 
be written as 
\[
\bigl(R_n(u)\bigr)_{ij}\, = \, \left( \tHe{U_\al^t, V_\al^t+\la^\al_n(i)}{C_\al^t}{u} 
n! \, [C_\al^t, U_\al^t, V_\al^t+\la^\al_n(i)]_n^{-1} e_i\right)^t_j
\]
for all $\al\in\R$. 
\end{thm}

Note that the left hand side is independent of $\al$, which is not obvious for the right hand side.

\begin{proof} Let us first assume that $\al$ is irrational, so that the result follows from the 
considerations in this section using that the $i$-th row of $R_n(u)$ is a polynomial of (precise) 
degree $n$. The constant follows from monocity of $R_n$, so that $(R_n(u))_{ii}=u^n$. 

Note that the left hand side is independent of $\al$, and the right hand side is continuous  in 
$\al$. 
Hence the result follows for $\al\in \R$. 
\end{proof}

In the scalar case $\ell=0$ Theorem \ref{thm:monicRnasMVHF} reduces to 
\begin{equation}\label{eq:scalarcaseHGform}
R_n(u) \, = \, (-4)^{-n} (n+1)\, \rFs{2}{1}{-n, \, n+2}{\frac32}{u},
\end{equation}
which is the well-known hypergeometric expression for the monic Chebyshev polynomials,
see \cite[\S 2.5]{AndrAR}, \cite[(4.5.21)]{Isma}, \cite[(1.8.31)]{KoekS}.

\section{Three-term recurrence relation}\label{sec:3termrecurrencerelation}

Matrix-valued orthogonal polynomials satisfy a three-term recurrence relation, see e.g. 
\cite{DamaPS}, \cite{GrunT}.
In \cite[Thm.~4.8]{KRvP} we have determined the three-term recurrence relation for the closely 
related matrix-valued orthogonal polynomials explicitly in terms of Clebsch-Gordan coefficients. 
The matrix entries of the matrices occurring in the three-term recurrence relation have been given 
explicitly as sums of products of Clebsch-Gordan coefficients. The purpose of this section is to 
give simpler expressions for the monic matrix-valued orthogonal polynomials using the explicit 
expression in terms of Tirao's matrix-valued hypergeometric functions as established in Theorem 
\ref{thm:monicRnasMVHF}. 

From general theory the monic orthogonal polynomials $R_n\colon \R\to M_N(\C)$  satisfy a 
three-term recurrence relation
$uR_n(u)=R_{n+1}(u)+X_{n}R_n(u)+Y_nR_{n-1}(u)$, $n\geq 0$, 
where $R_{-1}=0$ and $X_n,Y_n\in M_{2\ell+1}(\C)$ are matrices depending on $n$ and not on $x$. 
Lemma \ref{lem:recurrence.coeff} should be compared to \cite[Lemma~2.6]{DamaPS}.

\begin{lem}\label{lem:recurrence.coeff} 
Let $\{R_n\}_{n\geq 0}$ be the sequence of monic orthogonal polynomials and write
$R_n(u)=\sum_{k=0}^n R^n_k \, u^k$, $R^n_k \in M_N(\C)$, and $R^n_n=I$.  
Then the coefficients $X_n$, $Z_n$ of the three-term
recurrence relation are given by 
\[
X_n\, =\, R^n_{n-1}-R^{n+1}_n,\qquad Y_n\, =\, R^n_{n-2}-R^{n+1}_{n-1}-X_nR^n_{n-1}. 
\]
\end{lem}

\begin{proof} Let $\langle\cdot, \cdot\rangle$ denote the matrix-valued inner product for which the 
monic polynomials are orthogonal. Using the three-term recursion, orthogonality relations and 
expanding the monic polynomial of degree $n+1$ gives
\begin{equation*}
\langle uR_n-X_{n}R_n-Y_nR_{n-1}, R_n\rangle\, =\, 
\langle R_{n+1}, R_n\rangle \, =\,  \langle u^{n+1}, R_n\rangle \, + \, R^{n+1}_n \langle u^{n}, R_n\rangle.
\end{equation*}
By the orthogonality relations the left hand side can be evaluated as
\begin{equation*}
\langle uR_n-X_{n}R_n-Y_nR_{n-1}, R_n\rangle\, =\, \langle u^{n+1}, R_n\rangle \, +\, R^n_{n-1}\langle u^{n}, R_n\rangle \, -\, X_n \langle R_n, R_n\rangle
\end{equation*}
and comparing the two right hand sides gives the required expression for $X_n$, since $\langle R_n, R_n\rangle= \langle u^n,  R_n\rangle$ is invertible. 

The expression for $Y_n$ follows by considering on the one hand
\begin{equation*}
\begin{split}
&\langle uR_n-X_{n}R_n-Y_nR_{n-1}, R_{n-1}\rangle\, =\, 
\langle R_{n+1}, R_{n-1}\rangle \, \\ 
=&\,  \langle u^{n+1}, R_{n-1}\rangle \, + \, R^{n+1}_n \langle u^{n}, R_{n-1}\rangle \, + \, R^{n+1}_{n-1} \langle u^{n-1}, R_{n-1}\rangle.
\end{split}
\end{equation*}
while on the other hand the left hand side also equals
\begin{equation*}
\begin{split}
&\langle u^{n+1}, R_{n-1}\rangle\, + \, R^n_{n-1} \langle u^n, R_{n-1}\rangle \, 
+ \, R^n_{n-2} \langle u^{n-1}, R_{n-1}\rangle\, \\ & \qquad  - \, X_n \langle u^n, R_{n-1}\rangle\,  -\, 
 \, X_n R^n_{n-1} \langle u^{n-1}, R_{n-1}\rangle\,
-\,Y_n\langle R_{n-1}, R_{n-1}\rangle
\end{split}
\end{equation*}
and using the expression for $X_n$ and cancelling common terms gives the required expression, since
$\langle R_{n-1}, R_{n-1}\rangle= \langle u^{n-1},  R_{n-1}\rangle$ is invertible. 
\end{proof}

In order to apply Lemma \ref{lem:recurrence.coeff} for the explicit  monic polynomials in this 
paper we need to calculate the coefficients, which is an application of Theorem 
\ref{thm:monicRnasMVHF}.

\begin{lem}\label{lem:coeff_Rnm1_Rn}
 Let $\{R_n\}_{n\geq 0}$ be the monic polynomials with respect to $Z$ on $[0,1]$. Then
\begin{equation*}
\begin{split}
R^{n}_{n-1}&=\sum_{j=0}^n \frac{jn}{4(n+j)} E_{j,j-1} -\sum_{j=0}^n \frac{n}{2} E_{j,j} +
\sum_{i=0}^n \frac{n(2\ell-j)}{4(2\ell-j+n)}E_{j,j+1} \\
R^{n}_{n-2}&=
\sum_{j=0}^n \frac{n(n-1)j(j-1)}{32(n+j)(n+j-1)} E_{j,j-2}
-\sum_{j=0}^n \frac{n(n-1)j}{8(n+j)} E_{j,j-1} \\
&\quad +\sum_{j=0}^n \frac{n(n-1)(3j^2-6\ell j-2n^2+n-4n\ell)}{16(n+j)(i-2\ell-n)} E_{j,j} \\
&\quad -\sum_{j=0}^n \frac{n(n-1)(2\ell-j)}{8(2\ell+n-j)} E_{j,j+1} \\
&\quad +\sum_{j=0}^n \frac{n(n-1)(2\ell-j)(2\ell-j-1)}{32(2\ell-j+n-1)(2\ell+n-j)}E_{j,j+2} 
\end{split}
\end{equation*}
\end{lem}

\begin{proof} We can calculate $R^{n}_{n-1}$ by considering the coefficients of $u^{n-1}$ using the 
expression in Theorem \ref{thm:monicRnasMVHF}. This gives 
\begin{multline*}
(R^n_{n-1})_{ij}\, =\,
\frac{n!}{(n-1)!} \left( [C_\al^t, U_\al^t, V_\al^t+\la_n^\al(i)]_{n-1}
[C_\al^t, U_\al^t, V_\al^t+\la_n^\al(i)]_{n}^{-1} e_i
\right)^t_j \, \\
= \, n \, \left(M_{n-1}^\al\bigl(\la_n^\al(i)\bigr)^{-1} \, (C_\al^t+n-1) e_i\right)^t_j
\end{multline*}
using the recursive definition \eqref{eq:MV-Pochhammer} of $[C_\al^t, U_\al^t, V_\al^t+\la_n^\al(i)]_{n}$. Note that 
$M_{n-1}^\al\bigl(\la_n^\al(i)\bigr)$ is indeed invertible by Lemma \ref{lem:eigenvalues_lambda} for irrational $\al$. 
The explicit expression of the right hand side gives the result after a straightforward 
computation, since the resulting matrix is tridiagonal.

We can calculate $R^n_{n-2}$ analogously, $(R^n_{n-2})_{ij} =$
\[
\, n(n-1) \, \left(M_{n-2}^\al\bigl(\la^\al_n(i)\bigr)^{-1} (C_\al^t+n-2) M_{n-1}^\al\bigl(\la_n^\al(i)\bigr)^{-1} \, (C_\al^t+n-1) e_i\right)^t_j
\]
and a straightforward but tedious calculation gives the result. Note that $R^n_{n-2}$ is a 
five-diagonal matrix, since it is the product of two tridiagonal matrices. 
\end{proof}

Note that even though we have used the additional degree of freedom $\al$ in the proof of Lemma 
\ref{lem:coeff_Rnm1_Rn}, the resulting expressions are indeed independent of $\al$. 

Now we are ready to obtain the coefficients in the recurrence relation satisfied by the polynomials 
$R_n$.

\begin{thm}\label{thm:three_term_for_Rm}
For any $\ell\in\frac{1}{2}\mathbb{N}$ the monic orthogonal polynomials
$R_n$ satisfy the three-term recurrence relation
\begin{equation*}
u\, R_{n}(u)\, =\, R_{n+1}(u)\, +\, X_{n}\, R_{n}(u)\, +\, Y_{n}\, R_{n-1}(u),
\end{equation*}
where the matrices $X_n$, $Y_n$ are given by
\begin{equation*}
\begin{split}
X_n&\, =\, -\sum_{i=0}^{2\ell} \left[ \frac{i^2E_{i,i-1} }{4(n+i)(n+i+1)} -\frac{E_{i,i}}{2}+
\frac{(2\ell-i)^2E_{i,i+1}}{4(2\ell+n-i)(2\ell+n-i+1)}\right ], \\
Y_n&\, =\, \sum_{i=0}^{2\ell} \frac{n^2(2\ell+n+1)^2}{16(n+i)(n+i+1)(2\ell+n-i)(2\ell+n-i+1)} E_{i,i}.
\end{split}
\end{equation*}
\end{thm}

\begin{proof} 
This is a straightforward computation using Lemma \ref{lem:recurrence.coeff} and Lemma \ref{lem:coeff_Rnm1_Rn}. The calculation of $X_n$ is straightforward from 
Lemma \ref{lem:recurrence.coeff} and Lemma \ref{lem:coeff_Rnm1_Rn}.  In order to calculate $Y_n$ we need 
$X_n R_{n-1}$. A calculation shows
\begin{multline*}
X_nR^n_{n-1}=
-\sum_{j=0}^{2\ell} \frac{nj^2(j-1)}{16(n+j-1)(n+j)(n+j+1)} E_{j,j-2} 
+ \\
\sum_{j=0}^{2\ell} \frac{nj(n+2j+1)}{8(n+j)(n+j+1)} E_{j,j-1} + \\
\sum_{j=0}^{2\ell}\hskip-.15truecm \left(
\frac{-i^2n(2\ell-i+1)}{16(n+i)(n+i+1)(2\ell-i+1+n)} \right. 
-\frac{n}{4}- \\ \left. \qquad\qquad 
\frac{(2\ell-i)^2(i+1)n}{16(2\ell-i+1+n)(2\ell-i+n)(n+i+1)} \right) E_{jj} \\
 \qquad\quad -\sum_{j=0}^{2\ell} \frac{n(2\ell-i)(4\ell-2j+n+1)}{8(2\ell+n-j)(2\ell+n-j+1)} E_{j,j+1} \\
+\sum_{j=0}^{2\ell} \frac{n(2\ell-j)^2(2\ell-j+1)}{16(2\ell-j+n-1)(2\ell-j+n)(2\ell-j+n+1)}E_{j,j+2}
\end{multline*}
Now Lemma \ref{lem:recurrence.coeff} and a computation show that $Y_n$ reduces to a tridiagonal matrix.
\end{proof}

Now \eqref{eq:defpolsRn} and Theorem \ref{thm:three_term_for_Rm} give the three-term recurrence 
\begin{equation}\label{eq:3trrmonicPns}
x\, P_n(x)\, =\, P_{n+1}(x)\,+\,(1-2X_n)\, P_n(x)\, +\, 4Y_nP_{n-1}(x)
\end{equation}
for the monic orthogonal polynomials with respect to the matrix-valued weight $W$ on $[-1,1]$. The 
case $\ell=0$ corresponds to the three-term recurrence for the monic Chebyshev polynomials $U_n$. 
Note moreover, that $\lim_{n\to \infty} X_n = \frac12$ and $\lim_{n\to \infty} Y_n = \frac{1}{16}$, 
so that the monic matrix-valued orthogonal polynomials fit in the Nevai class, see \cite{DuraJAT}. 
Note the matrix-valued orthogonal polynomials $P_n$ in this paper are considered as matrix-valued 
analogues of the Chebyshev polynomials of the second kind, because of the group theoretic 
interpretation \cite{KRvP} and
Section \ref{sec:grouptheoreticderivation}, but that these polynomials are not matrix-valued 
Chebyshev polynomials in the sense of \cite[\S 3]{DuraJAT}.

Using the three-term recurrence relation \eqref{eq:3trrmonicPns} and \eqref{eq:ortho-monicP} we get 
\begin{equation}
\begin{split}
4Y_{n}H_{n-1} \, &= \, \int_{-1}^1 xP_n(x)W(x)P_{n-1}(x)^\ast dx\, \\ &= \, 
\int_{-1}^1 P_n(x)W(x)(xP_{n-1}(x))^\ast dx\, \, =\, H_n
\end{split}
\end{equation}
analogous to the scalar-valued case. Since $H_0$ is determined in \eqref{eq:H0explicitvalue} we obtain $H_n$. 

\begin{cor}\label{cor:squarednormHn} The squared norm matrix $H_n$ is
\[
(H_n)_{ij}\, = \, \de_{ij} \frac{\pi}{2} \frac{(n!)^2\, (2\ell+1)_{n+1}^2}{(i+1)_n^2 \, (2\ell-i+1)_n^2}
\frac{2^{-2n}}{(n+i+1)(2\ell-i+n+1)}
\]
and $JH_nJ=H_n$. 
\end{cor}

In \cite[Thm.~4.8]{KRvP} we have stated the three-term recurrence relation for the polynomials 
$Q^\ell_n(a)$, $a\in A_\ast$, see also Section \ref{sec:grouptheoreticderivation} of this paper. 
Apart from a relabeling of the orthonormal basis the monic polynomials corresponding to $Q_n^\ell$ 
are precisely the polynomials $P_n$, see \cite[\S 6.2, (6.4)]{KRvP} for the precise identification
\begin{equation}\label{eq:def_polynomials_Pd}
P_d(x)_{n,m}=\Upsilon_d^{-1} Q_d(a_{\arccos{x}})_{-\ell+n,-\ell+m}, \quad n,m\in\{0,1,\ldots,2\ell\},
\end{equation}
see also Section \ref{sec:grouptheoreticderivation}, where $\Upsilon_d$ in \eqref{eq:def_polynomials_Pd} is the leading coefficient of $Q^\ell_d$. 

\begin{cor}\label{cor:thm-three_term_for_Rm} The polynomials $Q^\ell_n$ as in \cite[\S 4]{KRvP} satisfy the recurrence 
\[
\phi(a)\, Q_n^\ell(a)\, =\, A_n\, Q_{n+1}^\ell(a) \, +\,  B_n \, Q_{n}^\ell(a) 
\, +\, C_n\,Q_{n-1}^\ell(a)
\]
where 
\begin{gather*}
A_n\, =\, \sum_{q=-\ell}^\ell \frac{(n+1)(2\ell+n+2)}{2(\ell-q+n+1)(\ell+q+n+1)} E_{q,q},  \\
B_n\, =\, \sum_{q=-\ell}^\ell  \frac{(\ell-q+1)(\ell+q)}{2(\ell-q+n+1)(\ell+q+n+1)}E_{q,q-1} \\
\qquad\qquad\qquad\qquad 
+ \sum_{q=-\ell}^\ell  \frac{(\ell+q+1)(\ell-q)}{2(\ell-q+n+1)(\ell+q+n+1)}E_{q,q+1} \\
C_n=\sum_{q=-\ell}^\ell \frac{n(2\ell+n+1)}{2(n+q+\ell+1)(n-q+\ell+1)} E_{q,q}.
\end{gather*}
\end{cor}

\begin{proof} We use $A_n=\Upsilon_n \Upsilon_{n+1}^{-1}$, $B_n=\Upsilon_n(1-2X_n) \Upsilon_{n}^{-1}$,
$C_n=\Upsilon_n(4Y_n) \Upsilon_{n-1}^{-1}$ and Theorem \ref{thm:three_term_for_Rm} to obtain the result from a straighforward computation. 
\end{proof}

Recall from \cite[Thm.~4.8, (3.2)]{KRvP} that the matrix entries of the matrices $A_n$, $B_n$ and 
$C_n$ are explicitly known as a square of a double sum with summand the product of four 
Clebsch-Gordan coefficients, hence Corollary \ref{cor:thm-three_term_for_Rm} leads to an explicit 
expression for this square. 

\section{The matrix-valued orthogonal polynomials related to Gegenbauer and Racah  polynomials}\label{sec:MVOPintermsofGegenbauerpols}

The LDU-decomposition of the weight $W$ of Theorem \ref{thm:LDUdecompW} has the weight functions of 
the Gegenbauer polynomials in the diagonal $T$, so we can expect a link between the matrix-valued 
polynomials $P_n(x)L(x)$ and
the Gegenbauer polynomials. We cannot do this via the orthogonality relations and the weight 
function, since the matrix $L$ also depends on $x$. Instead we use an approach based on the 
differential operators $\tilde D$ and $\tilde E$ of Section \ref{sec:differentialoperators}, and 
because of the link to the matrix-valued hypergeometric differential operator as in Theorem 
\ref{thm:monicRnasMVHF} we switch to the matrix-valued orthogonal polynomials $R_n$ and $x=1-2u$. 
It turns out that the matrix entries of $P_n(x)L(x)$ can be given as a product of a Racah 
polynomial times a Gegenbauer polynomial, see Theorem \ref{thm:cRnasGegenbauertimesRacah}. 

We use the differential operators $D$ and $E$ of Corollary \ref{cor:thmdifferentialoperatorsP}, and 
as in Section \ref{sec:HGfunctions} it is handier to work with the second-order differential 
operator $D_{-2\ell} = D-2\ell E$. By Theorem \ref{thm:LDUdecompW} we have $W(x) = L(x) T(x) 
L(x)^t$, hence 
$Z(u) = L(1-2u) T(1-2u) L(1-2u)^t$. For this reason we look at the differential operator conjugated 
by $M(u) = L(1-2u)$. 

In general, for $D= \frac{d^2}{du^2}F_2(u) + \frac{d}{du}F_1(u) + F_0(u)$ a second order matrix-
valued differential operator, conjugation with the matrix-valued function $M$, which we assume 
invertible for all $u$, gives
\begin{multline*}
M^{-1}DM \, =\, \frac{d^2}{du^2}M^{-1}F_2M + \frac{d}{du}\left( M^{-1}F_1M + 2\frac{dM^{-1}}{du}F_0M\right) \\
+ \left( M^{-1}F_0M +\frac{dM^{-1}}{du}F_1M + \frac{d^2M^{-1}}{du^2}F_0M\right).
\end{multline*}
Note that differentiating $M^{-1}M=I$ gives $\frac{dM^{-1}}{du}=-M^{-1}\frac{dM}{du}M^{-1}$, and similarly we find 
$\frac{d^2M^{-1}}{du^2}=-M^{-1}\frac{d^2M}{du^2}M^{-1} + 2M^{-1}\frac{dM}{du}M^{-1}\frac{dM}
{du}M^{-1}$. We are investigating the possibility of $M^{-1}DM$ being a diagonal matrix-valued 
differential operator. We now assume that $F_2(u) = u(1-u)$, so that $M^{-1}F_2M=u(1-u)$. A 
straightforward calculation using this assumption and the calculation of the derivatives of 
$M^{-1}$ shows that $M^{-1}DM = u(1-u) \frac{d^2}{du^2} + \frac{d}{du}T_1 + T_0$ with $T_0$ and 
$T_1$ matrix-valued functions if and only if the following equations \eqref{eq:diagMVDOcond1}, 
\eqref{eq:diagMVDOcond2} hold:
\begin{equation}\label{eq:diagMVDOcond1}
F_0M\,-\, \frac{dM}{du}T_1\, -\,  u(1-u)\, \frac{d^2M}{du^2}\, = \, MT_0
\end{equation}
\begin{equation}\label{eq:diagMVDOcond2}
F_1M \, - \, 2u(1-u) \frac{dM}{du}\, = \, MT_1.
\end{equation}
Of course, $T_0$ and $T_1$ need not be diagonal in general, but this is the case of interest. 

\begin{prop}\label{prop:OperatorDconjugatedbyM}
The differential operator $\cD = M^{-1} D_{-2\ell} M$ is the 
diagonal differential operator 
\[
\cD\, =\,  u(1-u)\frac{d^2}{du^2}
\, +\, \left(\frac{d}{du}\right) \, T_1(u)+ T_0,
\]
where
\[
T_1(u) = \frac12 T_1^1-  u\, T_1^1, \quad T_1^1\, = \, \sum_{i=0}^{2\ell} (2i+3)E_{i,i}, 
\quad T_0 \, =\, \sum_{i=0}^{2\ell} (2\ell-i)(2\ell+i+2) E_{i,i}
\]
Moreover, $\cR_n(u)\, = \, R_n(u)M(u)$ satisfies 
\[
\cR_n\, \cD\, = \, \La_n(\cD)\, \cR_n, \qquad  \La_n(\cD)\, = \, \La_n(D) -2\ell\, \La_n(E).
\]
\end{prop}

The proof shows that $M^{-1}D_\al M$ can only be a diagonal differential operator for $\al=-2\ell$. 
Note that $\cD$ is a matrix-valued differential operator as considered by Tirao, see Remark 
\ref{rmk:TiraoMVHF} and \cite{TiraPNAS}, and diagonality of $\cD$ implies that the matrix-valued 
hypergeometric ${}_2H_1$-series can be given explicitly in terms of (usual) hypergeometric series. 
In particular, we find as in the proof of Theorem \ref{thm:monicRnasMVHF} that 
\begin{equation}\label{eq:cRasMVHF}
\begin{split}
&\bigl( \cR_n(u)\bigl)_{kj}  =  \left( \tHe{T^1_1, \, \la_n(k)-T_0}{\frac12 T^1_1}{u}\, v\right)^t_j, \\ 
&v_k  =  \bigl( \cR_n(0)\bigl)_{kj}, \quad 
\la_n(k) = \La_n(\cD)_{kk},
\end{split}
\end{equation}
since the condition $\si(\frac12 T_1^1)\not\subset -\N$ is satisfied. 

\begin{proof} Consider $D_\al=D+\al E$, so that $F_2(u) = u(1-u)$ and the above considerations 
apply and $F_1(u)= C_\al-uU_\al$, and $F_0=-V_\al$. We want to find out if we can obtain matrix-
valued functions $T_1$ and $T_0$ satisfying 
\eqref{eq:diagMVDOcond1}, \eqref{eq:diagMVDOcond2} for this particular $F_1$, $F_2$ and 
$M(u)=L(1-2u)$. Since $F_0$ is diagonal, and assuming that $T_0$, $T_1$ can be taken diagonal it is 
clear that taking the $(k,l)$-th entry of \eqref{eq:diagMVDOcond1}  leads to 
\begin{equation}\label{eq:diagMVDOcond1-A}
(F_0)_{kk} M_{kl} - \frac{dM_{kl}}{du} (T_1)_{ll} - u(1-u) \frac{d^2M_{kl}}{du^2} \, = \, M_{kl} (T_0)_{ll}. 
\end{equation}
By Theorem \ref{thm:LDUdecompW} we have $M_{kl}=0$ for $l>k$ and for $l\leq k$
\[
M_{kl}(u) \, = \, \binom{k}{l}
\rFs{2}{1}{l-k, k+l+2}{l+\frac32}{u}
\]
so that \eqref{eq:diagMVDOcond1-A} has to correspond to the second order differential operator 
\begin{equation}\label{eq:HGdiffopscalar}
u(1-u) f''(u) \, + \, \bigl( c- (a+b+1)u\bigr) f'(u) - abf(u)=0, \quad f(u) \, = \, \rFs{2}{1}{a,b }{c}{u}
\end{equation}
for the hypergeometric function. This immediately gives   
\[
(T_1)_{ll}\, = \, l+\frac32 -(2l+3)u, \qquad 
(T_0)_{ll}-(F_0)_{kk} \, = \, k^2+2k-(l^2+2l). 
\]
Since $(F_0)_{kk} = (-V_\al)_{kk} =  -k^2+(2\ell +\al\frac{(2\ell+2)}{2\ell})k -\al(2\ell+2)$, this is only possible for $\al=-2\ell$, and in that case 
\begin{equation}\label{eq:diagD-T0T1explicit}
(T_1)_{ll}\, = \, l+\frac32 -(2l+3)u, \qquad 
(T_0)_{ll} \, = \, -l^2-2l+2\ell(2\ell+2). 
\end{equation}

It remains to check that for $\al=-2\ell$ the condition \eqref{eq:diagMVDOcond2} is valid with the 
explicit values \eqref{eq:diagD-T0T1explicit}. For $\al=-2\ell$ the matrix-valued function $F_1$ is 
lower triangular instead of tridiagonal, so that \eqref{eq:diagMVDOcond2} is an identity in the 
subalgebra of lower triangular matrices. With the explicit expression for $M$ we have to check that 
\[
\bigl( (\frac32 +k) - u(3+2k)\bigr) M_{kl} - kM_{k-1,l}\, - \, 2 u(1-u)\frac{dM_{kl}}{du}\, = \, 
M_{kl}\bigl( \frac32 +l- u(3+2l)\bigr) 
\]
which can be identified with the identity 
\[
(1-x^2)\, \frac{dC_{k-l}^{(l+1)}}{dx}(x) \, = \, (k+l+1) C_{k-l-1}^{(l+1)}(x) - x(k-l) C^{(l+1)}_{k-l}(x).
\]
In turn, this identity can be easily obtained from \cite[(22.7.21)]{AbraS} or from \cite[(4.5.3), (4.5.7)]{Isma}.
\end{proof}

Since $R_n$ and $M$ are polynomial, Proposition \ref{prop:OperatorDconjugatedbyM} and the explicit 
expression for the eigenvalue matrix in Corollary \ref{cor:thmdifferentialoperatorsP} imply that 
$(\cR_n)_{kj}$ is a polynomial solution to 
\begin{multline*}
u(1-u) f''(u) + \bigl( (j+\frac32)- u(2j+3)\bigr)\, f'(u) + (2\ell-j)(2\ell+j+2) f(u) \\\, = \,
\Bigl( -n(n-1)-n(2\ell+3)+k(2\ell-k) +2n(\ell-k)-(2\ell+2)(k-2\ell)\Bigr)\, f(u)
\end{multline*}
which can be rewritten as 
\begin{multline*}
u(1-u) f''(u) + \bigl( (j+\frac32)- u(2j+3)\bigr)\, f'(u) -(j-k-n)(n+k+j+2) f(u)  =  0
\end{multline*}
which is the hypergeometric differential operator for which the polynomial solutions are uniquely determined up to a constant. This immediately gives
\begin{equation}\label{eq:defconstantsckjn}
\cR_n(u)_{kj} \, = \, c_{kj}(n) \, \rFs{2}{1}{j-k-n,\, n+k+j+2}{j+\frac32}{u}.
\end{equation}
for $j-k-n\leq 0$ and $\cR_n(u)_{kj} = 0$ otherwise. The case $n=0$ corresponds to Theorem \ref{thm:LDUdecompW} and we obtain $c_{kj}(0) = \binom{k}{j}$. 
It remains to determine the constants $c_{kj}(n)$ in \eqref{eq:defconstantsckjn}. 

First, switching to the variable $x$, we find
\begin{equation}\label{eq:cPn=constantGegenbauerpols}
\bigl( \cP_n(x)\bigr)_{kj} \, = \, \bigl(P_n(x)L(x)\bigr)_{kj}\, =\, (-2)^n\, c_{kj}(n)\, \frac{(n+k-j)!}{(2j+2)_{n+k-j}} C^{(j+1)}_{n+k-j}(x)
\end{equation}
so that by \eqref{eq:cPn=constantGegenbauerpols} the orthogonality relations \eqref{eq:ortho-monicP} and 
\eqref{eq:orthorelGegenbauerpols} give
\begin{equation*}
\begin{split}
\de_{nm} (H_n)_{kl} \, = \, (-2)^{n+m}\,  &\sum_{j=0}^{2\ell\wedge(n+k)} c_{kj}(n) \overline{c_{lj}(m)} c_j(\ell) \\
& \qquad \times \frac{(n+k-j)!}{(2j+2)_{n+k-j}} \, 
\de_{n+k,m+l} 
\frac{\sqrt{\pi}\, \Ga(j+\frac32)}{(n+k+1)\, j!}.
\end{split}
\end{equation*}
Using the explicit value for $c_j(\ell)$ as in Theorem \ref{thm:LDUdecompW} and Corollary 
\ref{cor:squarednormHn} we find  orthogonality relations for the coefficients $c_{kj}(n)$:
\begin{equation}\label{eq:orthorelckjn}
\begin{split}
(H_n)_{kk} \, 2^{-2n}\, \de_{nm}\, =&   
\sum_{j=0}^{2\ell\wedge (n+k)} c_{kj}(n) \overline{c_{k+m-n,j}(m)} \\ 
&\qquad 
\times\frac{(j!)^2\, (2j+1)\, (2\ell+j+1)!\, (2\ell-j)!\, (n+k-j)!}{(n+k+j+1)!\, (n+k+1)\, (2\ell)!^2}
\end{split}
\end{equation}

Note that we can also obtain recurrence relations for the coefficients $c_{kj}(n)$ using the three-term recurrence relation of Theorem \ref{thm:three_term_for_Rm}. 

\begin{thm}\label{thm:cRnasGegenbauertimesRacah} The polynomials $\cR_n(u) = R_n(u) M(u)$ satisfy 
\begin{multline*}
\cR_n(u)_{kj} = c_{k,0}(n) (-1)^j \frac{(-2\ell)_j\, (-k-n)_j}{j!\, (2\ell+2)_j}
\, \rFs{4}{3}{-j, j+1, -k, -2\ell-n-1}{1, -k-n, -2\ell}{1} \\ \, \times
\rFs{2}{1}{j-k-n,\, n+k+j+2}{j+\frac32}{u}
\end{multline*}
with  $\cR_n(u)_{kj} = 0$ for $j-k-n> 0$ and
\[
c_{k,0}(n)\, = \,  (-1)^n 4^{-n} \frac{n!\, (2\ell+2)_n}{(k+1)_n\, (2\ell-k+1)_n}
\]
\end{thm}

We view Theorem \ref{thm:cRnasGegenbauertimesRacah} is an extension of Theorem 
\ref{thm:LDUdecompW}, but Theorem \ref{thm:LDUdecompW} is instrumental in the proof of Theorem 
\ref{thm:cRnasGegenbauertimesRacah}. Since the inverse of $M(u)$, or of $L(x)$, does not seem to 
have a nice explicit expression we do not obtain an interesting expression for the matrix elements 
of the matrix-valued monic orthogonal polynomials $R_n(u)$ or of $P_n(x)$. Note also that the case 
$\ell=0$ gives back the hypergeometric representation of the Chebyshev polynomials of the second 
kind $U_n$, see \eqref{eq:scalarcaseHGform}. 

Comparing with \eqref{eq:defRacahpols} we see that the ${}_4F_3$-series in Theorem 
\ref{thm:cRnasGegenbauertimesRacah} can be viewed as a Racah polynomial $R_k(\la(j);-2\ell-1,-k-
n-1,0,0)$, respectively $R_{k+n-m}(\la(j);-2\ell-1,-k-n-1,0,0)$, see \eqref{eq:defRacahpols}, where 
the $N$ of the Racah polynomials equals $2\ell$ in case $2\ell\leq k+n$ and $N$ equals $k+n$ in 
case $2\ell\geq k+n$. Using the first part of Theorem \ref{thm:cRnasGegenbauertimesRacah} we see 
that the orthogonality relations 
\eqref{eq:orthorelckjn} lead to 
\begin{equation}\label{eq:orthorelckjn-Racahcase} 
\begin{split}
&(H_n)_{kk} \, 2^{-2n}\, \de_{nm}\, =  \frac{\pi}{2}\frac{2\ell+1}{(n+k+1)^2} |c_{k0}(n)|^2 
\sum_{j=0}^{2\ell\wedge (n+k)} \frac{(2j+1)\, (-2\ell)_j \, (-n-k)_j}{(2\ell+2)_j\, (n+k+2)_j}
\\ &  \times R_k(\la(j);-2\ell-1,-k-n-1,0,0)\, R_{k+n-m}(\la(j);-2\ell-1,-k-n-1,0,0),
\end{split}
\end{equation}
which corresponds to the orthogonality relations for the corresponding Racah polynomials, see \cite[p.~344]{AndrAR}, \cite[\S 1.2]{KoekS}. From this we find that the sum  
in \eqref{eq:orthorelckjn-Racahcase} equals 
\[
\de_{nm}\, \frac{(2\ell+1)(n+k+1)}{(2\ell+1+n-k)}.
\]
Hence, 
\begin{equation}\label{eq:ck0(n)absolutevalue}
\begin{split}
|c_{k0}(n)|^2 \, &= \, (H_n)_{kk} \, 2^{-2n}\, \frac{2}{\pi}\frac{(n+k+1)(2\ell+1+n-k)}{(2\ell+1)^2} \, 
\\&=\, 4^{-2n} \frac{(n!)^2\, (2\ell+2)_n^2}{(k+1)_n^2\, (2\ell-k+1)_n^2}
\end{split}
\end{equation}
using Corollary \ref{cor:squarednormHn}.

We end this section with the proof of Theorem \ref{thm:cRnasGegenbauertimesRacah}. The idea of the 
proof is to obtain a three-term recurrence for the coefficients $c_{kj}(n)$ with explicit initial 
conditions, and to compare the resulting three-term recurrence with well-known recurrences for 
Racah polynomials, see \cite{AndrAR}, \cite{Isma}, \cite{KoekS}. The three-term recurrence relation 
is obtained using the first-order differential operator $E$ and the fact that the $\cR_n$, being 
analytic eigenfunctions to $\cD$, are completely determined by the value at $0$, see Remark 
\ref{rmk:TiraoMVHF}. 

\begin{proof}[Proof of Theorem \ref{thm:cRnasGegenbauertimesRacah}] Since the matrix-valued 
differential operators $D$ and $E$ commute and have the matrix-valued orthogonal polynomials $R_n$ 
as eigenfunctions by 
Corollary \ref{cor:thmdifferentialoperatorsP}, we see that $\cE = M^{-1} EM$ satisfies 
\begin{equation}\label{eq:propertiescE}
\cE\, \cR_n \, = \, \La_n(\cE)\, \cR_n, \quad \La_n(\cE) \, = \, \La_n(E), \qquad \cE\, \cD\, =\, \cD\, \cE
\end{equation}
Moreover, in the same spirit as the proof of Proposition \ref{prop:OperatorDconjugatedbyM} we 
obtain
\begin{equation}\label{eq:explicitcE=MinvEM}
\begin{split}
\cE\, &= \, \left(\frac{d}{du}\right) S_1(u) + S_0(u), \\ 
S_1(u)\, &=\, u(1-u)\sum_{i=0}^{2\ell}  \frac{i^2(2\ell+i+1)}{\ell(2i-1)(2i+1)} \, E_{i,i-1}
\, -\, \sum_{i=0}^{2\ell}   \frac{(2\ell-i)}{4\ell}  \, E_{i,i+1},\\
S_0(u)\, &= \, (1-2u)\sum_{i=0}^{2\ell}  \frac{i^2(2\ell+i+1)}{2\ell(2i-1)}   \,E_{i,i-1}
\,+\, \sum_{i=0}^{2\ell}   \frac{i(i+1)-4\ell(\ell+1)}{2\ell} \, E_{i,i}
\end{split}
\end{equation}
by a straightforward calculation.   

Define the vector space of (row-)vector valued functions 
\[
V(\la)=\{ F\text{ analytic at }u=0\, \mid \, F\cD\, =\, \la F \}, 
\]
and $\nu \colon V(\la) \to \C^{2\ell+1}$, $F\mapsto F(0)$, is an isomorphism, see Remark \ref{rmk:TiraoMVHF} and \cite{TiraPNAS}. 
Because of \eqref{eq:propertiescE} we have 
the following commutative diagram
\begin{equation*}
\begin{CD}
V(\la) @ > \cE >>V(\la) \\ @ V \nu VV @ VV \nu V \\
\C^{2\ell+1} @> N(\la) >> \C^{2\ell+1}  
\end{CD}
\end{equation*}
with $N(\la)$ a linear map. In order to determine $N(\la)$ we note that $F\in V(\la)$ can be written as, cf \eqref{eq:cRasMVHF}, 
\[
F_j(u) = 
\left( \tHe{T^1_1, \, \la-T_0}{\frac12 T^1_1}{u}\, F(0)^t\right)^t_j, 
\]
so that $\frac{dF_j}{du}(0) = F(0)(\la-T_0)(\frac12 T_1^1)^{-1}$ by construction of the ${}_2H_1$-series, see Remark \ref{rmk:TiraoMVHF}. 
Now \eqref{eq:propertiescE} gives 
\begin{equation*}
N(\la) \, = \, (\la-T_0)(\frac12 T_1^1)^{-1}S_1(0) + S_0(0)
\end{equation*}
acting from the right on row-vectors from $\C^{2\ell+1}$.

By Proposition \ref{prop:OperatorDconjugatedbyM} we have that the $k$-th row $\bigr((\cR_n)_{kj}(\cdot)\bigl)_{j=0}^{2\ell}$ of $R_n$ is contained in 
$V(\la_n(k))$, see \eqref{eq:cRasMVHF}. On the other hand, the $k$-th row of $\cR_n$ is an eigenfunction of $\cE$ for the eigenvalue $\mu_n(k)=\La_n(\cE)_{kk}$. 
Since $\nu\Bigl(\bigr((\cR_n)_{kj}\bigl)_{j=0}^{2\ell}\Bigr) = (c_{kj}(n))_{j=0}^{2\ell}$ 
we see that the row-vector $c_k=(c_{kj}(n))_{j=0}^{2\ell}$ satisfies
$c_k N(\la_n(k)) = \mu_n(k)\, c_k$, which gives the recurrence relation 
\begin{multline}
\label{eq:recurrence_coefficients_ck}
 -\frac{(i+k+n+1)(i-k-n-1)(2\ell-i+1)}{(2i+1)}c_{k,i-1}(n) + (i(i+1)-4\ell(\ell+1))c_{k,i}(n)\\
+\frac{(i+1)^2(2\ell+i+2)}{(2i+1)}c_{k,i+1}(n)\, =\, (-2n(\ell-k)+(2\ell+2)(k-2\ell))c_{k,i}(n),
\end{multline}
with the convention $c_{k,-1}(n)=0$. Note that $c_{jk}(0)= \binom{k}{j}$ indeed satisfies 
\eqref{eq:recurrence_coefficients_ck}. 
Comparing \eqref{eq:recurrence_coefficients_ck} with the three-term recurrence relation for the 
Racah polynomials 
or the corresponding contiguous relation for balanced ${}_4F_3$-series, see e.g. \cite[p.~344]
{AndrAR}, \cite[\S 1.2]{KoekS}, gives
\begin{equation*} 
c_{kj}(n) \, = \, c_{k,0}(n) (-1)^j \frac{(-2\ell)_j\, (-k-n)_j}{j!\, (2\ell+2)_j}
\, \rFs{4}{3}{-j, j+1, -k, -2\ell-n-1}{1, -k-n, -2\ell}{1} 
\end{equation*}
and $c_{kj}(n) = 0$ for $j>k+n$. 

It remains to determine the constants $c_{k0}(n)$, and we have already determined their absolute 
values in \eqref{eq:ck0(n)absolutevalue} by matching it to the orthogonality relations for Racah 
polynomials. From the three-term recurrence relation Theorem \ref{thm:three_term_for_Rm} we see 
that the constants $c_{kj}(n)$ are all real, so it remains to determine the sign of $c_{k0}(n)$. 
Theorem \ref{thm:three_term_for_Rm} gives a three-term recurrence for $\cR_n(u)$, and taking the 
$(k,0)$-th matrix entry gives a polynomial identity in $u$ using \eqref{eq:defconstantsckjn}. Next 
taking the leading coefficient gives the recursion
\[
c_{k0}(n+1) = -\frac{(n+k+2)}{4(n+k+1)}\, c_{k0}(n) + \frac{(2\ell-k)^2}{4(2\ell+n-k)(2\ell+n-k+1)}\, c_{k+1,0}(n)
\]
and plugging in $c_{k0}(n) =\sgn(c_{k0}(n)) |c_{k0}(n)|$ and using the explicit value for $|c_{k0}(n)|$ gives 
\begin{multline*}
\sgn(c_{k0}(n+1))\, (n+1)(2\ell+n+2) =  \\
-\sgn(c_{k0}(n))(n+k+2)(2l-k+n+1) + \sgn(c_{k+1,0}(n))(2\ell-k)(k+1). 
\end{multline*}
This gives $\sgn(c_{k0}(n)) = \sgn(c_{k+1,0}(n))$ for the right hand side to factorise as in the left hand side, and then
$\sgn(c_{k0}(n+1)) = -\sgn(c_{k0}(n))$. Since $c_{k0}(0)=1$, we find $\sgn(c_{k0}(n))=(-1)^n$. 
\end{proof}

\begin{rem}\label{rmk:thmcRnasGegenbauertimesRacah} Theorem \ref{thm:cRnasGegenbauertimesRacah} 
can now be plugged into the three-term recurrence relation for $\cR_n$ of Theorem 
\ref{thm:three_term_for_Rm}, and this then gives a intricate three-term recurrence relation for 
Gegenbauer polynomials involving coefficients which consist of sums of two Racah polynomials. We 
leave this to the interested reader.
\end{rem}

\begin{rem}\label{rmk:partiallyotherproofthmcRnasGegenbauertimesRacah}
We sketch another approach to the proof of the value of $c_{k0}(n)$ by calculating the value
$c_{k,2\ell}(n)$ in case $k+n\geq 2\ell$ or $c_{k,n+k}(n)$ in case $k+n<2\ell$. For instance, in case 
$k+n\geq 2\ell$ we have
\[
(\cR_n(u))_{k,2\ell}\, = \, (R_n(u)M(u))_{k,2\ell} \, = \, (R_n(u))_{k,2\ell} \, = \, (R_n(u))_{2\ell-k,0}
\]
using that $M$ is a unipotent lower-triangular matrix-valued polynomial and the symmetry 
$JR_n(u)J=R_n(u)$, see \cite[\S 5]{KRvP}. Now the leading coefficient of the right hand side can be 
calculated using Theorem \ref{thm:monicRnasMVHF}, and combining with 
\eqref{eq:defconstantsckjn}, the value $c_{k,2\ell}(n)$ follows. Then the recurrence  
\eqref{eq:recurrence_coefficients_ck} can be used to find $c_{k0}(n)$. 
\end{rem}

\section{Group theoretic interpretation}\label{sec:grouptheoreticderivation}

The purpose of this section is to give a group theoretic derivation of Theorem \ref{thm:differentialoperatorsP} complementing the analytic derivation of \cite[\S 7]{KRvP}. For this we need to recall some of the results of \cite{KRvP}.

\subsection{Group theoretic setting of the matrix-valued orthogonal polynomials}\label{ssec:grouptheoreticderivationDOs}

In this subsection we recall the construction of the matrix-valued orthogonal polynomials and the 
corresponding weight starting from the pair $(\SU(2)\times\SU(2),\SU(2))$ and an 
$\SU(2)$-representation $T^{\ell}$. Then we discuss how the differential operators come into play and what 
their relation is with the matrix-valued orthogonal polynomials. The goal of this section is to 
provide a map of the relevant differential operators in the group setting to the relevant 
differential operators for the matrix-valued orthogonal polynomials in Theorem 
\ref{def:to dif op pol}.

Let $U=\SU(2)\times\SU(2)$ and $K=\SU(2)$ diagonally embedded in $U$. Note that $K$ is the set of 
fixed points of the involution $\theta:U\to U:(x,y)\mapsto(y,x)$. The irreducible representations 
of $U$ and $K$ are denoted by $T^{\ell_{1},\ell_{2}}$ and $T^{\ell}$ as is explained in \cite[\S2]
{KRvP}. The representation space of $T^{\ell}$ is denoted by $H^{\ell}$ which is a 
$2\ell+1$-dimensional vector space. If $T^{\ell}$ occurs in $T^{\ell_{1},\ell_{2}}$ upon restriction to $K$ 
we defined the spherical function $\Phi^{\ell}_{\ell_{1},\ell_{2}}$ in \cite[Def.~2.2]{KRvP} as the 
$T^{\ell}$-isotypical part of the matrix $T^{\ell_{1},\ell_{2}}$. Let $A\subset U$ be the subgroup
\[A=\left\{a_{t}=\left(\left(\begin{array}{cc}e^{it/2}&0\\0&e^{-it/2}\end{array}\right),
\left(\begin{array}{cc}e^{-it/2}&0\\0&e^{it/2}\end{array}\right)\right), 0\le t<4\pi\right\}\]
and let $M=Z_{K}(A)$. Recall the decomposition $U=KAK$, \cite[Thm.~7.38]{KnappLGBI}. The restricted 
spherical functions $\Phi^{\ell}_{\ell_{1},\ell_{2}}|_{A}$ take values in $\mathrm{End}_{M}
(H^{\ell})$, see \cite[Prop.~2.3]{KRvP}. Since $\mathrm{End}_{M}(H^{\ell})\cong\C^{2\ell+1}$ this 
allows allows us to view the restricted spherical functions as being $\C^{2\ell+1}$-valued. The 
parametrization of the $U$-representations that contain $T^{\ell}$ indicates how to gather the 
restricted spherical functions. Following \cite[Fig.~3]{KRvP} we write $(\ell_{1},
\ell_{2})=\zeta(d,h)$ with $\zeta(d,h)=(\frac{1}{2}(d+\ell+h),\frac{1}{2}(d+\ell-h))$. Here 
$d\in\mathbb{N}$ and $h\in\{-\ell,-\ell+1,\ldots,\ell\}$. We recall the definition of the full 
spherical functions of type $\ell$, \cite[Def.~4.2]{KRvP}.

\begin{defin}\label{def: full sf} The full spherical function of type $\ell$ and degree $d$ is 
the matrix-valued function $\Phi^{\ell}_{d}:A\to\mathrm{End}(\C^{2\ell+1})$ whose $j$-th row is the 
restricted spherical function $\Phi^{\ell}_{\ell_{1},\ell_{2}}$ with $(\ell_{1},
\ell_{2})=\zeta(d,j)$. 
\end{defin}

The full spherical function of degree zero has the remarkable property of being invertible on the 
subset $A_{reg}:=\{a_{t}|t\in(0,\pi)\cup(\pi,2\pi)\cup(2\pi,3\pi)\cup(3\pi,4\pi)\}$, which was 
first proved by Koornwinder \cite[Prop.~3.2]{Koornwinder85}. The invertibility follows also from 
Corollary \ref{cor:det-thmLDUdecompW}. Let $\phi=\Phi^{0}_{1/2,1/2}$ be the minimal non-trivial 
zonal spherical function \cite[\S3]{KRvP}. Together with the recurrence relations for the full 
spherical functions with $\phi$ \cite[Prop.~3.1]{KRvP} this gives rise to the full spherical 
polynomials \cite[Def.~4.4]{KRvP}.

\begin{defin}\label{def: sf pol} The full spherical polynomial $Q_{d}^{\ell}:A\to\mathrm{End}
(\C^{2\ell+1})$ is defined by $Q_{d}^{\ell}(a)=(\Phi_{0}^{\ell}(a))^{-1}\Phi_{d}^{\ell}(a)$. 
\end{defin}

The name full spherical polynomial comes from the fact that the $Q_{d}^{\ell}$ are polynomials in 
$\phi$. The full spherical polynomials $Q_{d}^{\ell}$ are orthogonal with respect to 
\[\langle
Q,P\rangle_{V^{\ell}}=\int_{A}Q(a)V^{\ell}(a)P(a)da,\quad
V^{\ell}(a_{t})=(\Phi_{0}^{\ell}(a_{t}))^{*}\Phi_{0}^{\ell}(a_{t})\sin^{2}
t,
\]
see \cite[Cor. 5.7]{KRvP}. 

In \cite[\S5]{KRvP} we studied the weight functions $V^{\ell}$ extensively. It turns out that the 
matrix entries are polynomials in the function $\phi$, apart from the common factor $\sin t$. Upon 
changing the variable $x=\phi(a)$ we obtain the following system of matrix-valued orthogonal 
polynomials. 

\begin{defin}\label{def: sf pol in x}
Let $R_{d}^{\ell}:[0,1]\to\mathrm{End}(\C^{2\ell+1})$ be the polynomial defined by $R_{d}^{\ell}
(\phi(a))=Q_{d}^{\ell}(a)$. The degree of $R_{d}^{\ell}$ is $d$. The polynomials are orthogonal 
with respect to
\[\langle R,P\rangle_{W^{\ell}}=\int_{-1}^{1}R(x)W^{\ell}(x)P(x)\, dx,\]
where $W^{\ell}$ is defined by $W^{\ell}(\phi(a))d\phi=V^{\ell}(a)da$.
\end{defin}

The weight $W^{\ell}(x)$ from Definition \ref{def: sf pol in x} is the same as the weight defined 
in \eqref{eq:defmatrix_W} where we have to bear in mind that the basis is parametrised differently. 
The matrix-valued polynomials $R_{d}^{\ell}$ correspond to the family $\{P_{d}\}_{d\ge0}$ from 
Theorem \ref{thm:differentialoperatorsP} by means of making the $R^{\ell}_{d}$ monic. Given a 
system of matrix-valued orthogonal polynomials as in Definition \ref{def: sf pol in x} it is of 
great interest to see whether there are interesting differential operators. More precisely we define 
the algebra $\mathcal{D}(W^{\ell})$ as the algebra of differential operators that are self-adjoint 
with respect to the weight $W^{\ell}$ and that have the $R_{d}^{\ell}$ as eigenfunctions. We define 
a map that associates to a certain left invariant differential operator on the group $U$ an element 
in $\mathcal{D}(W)$.

Before we go into the construction we observe that the spherical functions may also be defined on 
the complexification $A^{\C}$, using Weyl's unitary trick. Indeed, all the representations that we 
consider are finite dimensional and unitary, so they give holomorphic representations of the 
complexifications $U^{\C}$ and $K^{\C}$.

A great part of the constructions that we are about to consider follows Casselman and Mili\v{c}i\'c 
\cite{CM1982}, where the differential operators act from the left. In this section we follow this 
convention, except that we transpose the results at the end in order to obtain the proof of Theorem 
\ref{thm:differentialoperatorsP} where the differential operators act from the right. 

Let $U(\lau^\C)$ be the universal enveloping algebra for the complexification $\lau^{\C}$ of the 
Lie algebra $\lau$ of the group $U=\SU(2)\times \SU(2)$. Let $\theta\colon U(\lau^\C) \to 
U(\lau^\C)$ be the flip on simple tensors extending the Cartan involution $\theta \colon 
\lasu(2)\times \lasu(2) \to \lasu(2)\times \lasu(2)$, $(X,Y)\mapsto (Y,X)$. Recall $\lak\cong 
\lasu(2)$ is the fixed-point set of $\theta$. Let $U(\lau^{\C})^{\lak^{\C}}$ denote the subalgebra 
of elements that commute with $\lak^{\C}$. Let $\laZ$ denote the center of $U(\lak^{\C})$.

\begin{lem}\label{lem: grp alg} $U(\lau^{\C})^{\lak^{\C}}\cong\laZ\otimes\laZ\otimes\laZ$.
\end{lem}

\begin{proof}
From \cite[Satz~2.1 and Satz~2.3]{Knop} it follows that 
$U(\lau^{\C})^{\lak^{\C}}\cong\laZ\otimes_{Z(\mathfrak{j})}(\laZ\otimes\laZ)$ where 
$\mathfrak{j}\subset\lau^{\C}$ is the largest ideal of $\lau^{\C}$ contained in $\lak^{\C}$. Since 
$\mathfrak{j}=0$ the result follows.
\end{proof}

\begin{prop}\label{prop: Csf eigenv}
The elements of the algebra $U(\lau^{\C})^{\lak^{\C}}$ have the spherical functions 
$\Phi^{\ell}_{\ell_{1},\ell_{2}}$ as eigenfunctions. This remains true when we extend 
$\Phi^{\ell}_{\ell_{1},\ell_{2}}$ to 
$U^\C$.
\end{prop}

\begin{proof}
See \cite[Thm.~6.1.2.3]{WarnerII}. The second statement follows from Weyl's unitary trick.
\end{proof}

The spherical functions $\Phi^{\ell}_{\ell_{1},\ell_{2}}$ have $T^{\ell}$-transformation behaviour:
\begin{eqnarray}\label{eq: trafo}
\Phi^{\ell}_{\ell_{1},\ell_{2}}(k_{1}uk_{2})=T^{\ell}(k_{1})\Phi^{\ell}_{\ell_{1},\ell_{2}}(u)T^{\ell}(k_{2})
\end{eqnarray}
for all $k_{1},k_{2}\in K$ and $u\in U$, see \cite[Def.~2.2]{KRvP}. Let $C(A)$ denote the set of 
continuous ($\C$-valued) functions on $A$. Casselman and Mili\v{c}i\'c \cite{CM1982} define the map
\[\Pi_{\ell}\colon U(\lau^{\C})^{\lak^{\C}}\to C(A)\otimes U(\laa^{\C})\otimes\mathrm{End}
(\mathrm{End}_{M}(H^{\ell}))\]
and prove the following properties \cite[Thm.~3.1,Thm.~3.3]{CM1982}. 

\begin{thm}\label{thm: radial part}
Let $F\colon U\to\mathrm{End}(H^{\ell})$ be a smooth function that satisfies \eqref{eq: trafo}. Then 
$(DF)|_{A}=\Pi_{\ell}(D)(F|_{A})$ for all $D\in U(\lau^{\C})^{\lak^{\C}}$. Moreover, $\Pi_{\ell}$ is an algebra homomorphism.
\end{thm}

We call $\Pi_{\ell}(D)$ the $T^{\ell}$-radial part of $D\in U(\lau^{\C})^{\lak^{\C}}$. In 
particular we have \[\Pi_{\ell}(D)(\Phi^{\ell}_{\ell_{1},\ell_{2}}|_{A})=\lambda_{D,\ell_{1},
\ell_{2}}^{\ell}\Phi^{\ell}_{\ell_{1},\ell_{2}}|_{A},\quad \lambda_{D,\ell_{1},
\ell_{2}}^{\ell}\in\C.
\]
Upon identifying $\mathrm{End}_{M}(H^{\ell})\cong\C^{2\ell+1}$ we observe that we may view 
$\Pi_{\ell}(D)$ as differential operator of the $\mathrm{End}(\C^{2\ell+1})$-valued functions that 
act on from the left. In particular, let $C(A,\mathrm{End}(\C^{2\ell+1}),T^{\ell})$ denote the 
vector space generated by the $\Phi_{d}^{\ell}, d\ge0$. The following lemma follows immediately 
from the construction.

\begin{lem} Let $D\in U(\lau^{\C})^{\lak^{\C}}$ be self-adjoint and consider $\Pi_{\ell}(D)$ as a 
differential operator acting on $C(A,\mathrm{End}(\C^{2\ell+1}),T^{\ell})$ from the left. Then $
\Pi_{\ell}(D)$ is self-adjoint for $\langle\cdot,\cdot\rangle_{V^{\ell}}$.
\end{lem}

\begin{defin}\label{def:to dif op pol}
Let $f\colon U(\lau^{\C})^{\lak^{\C}}\to\mathcal{D}(W^{\ell})$ be defined by sending $D$ to the 
conjugation of the differential operator $\Pi_{\ell}(D)$ by $\Phi_{0}^{\ell}$ followed by changing 
the
variable $x=\phi(a)$.
\end{defin}

The map $f\colon U(\lau^{\C})^{\lak^{\C}}\to\mathcal{D}(W^{\ell})$ is an algebra homomorphism. It gives 
an abstract description of a part of $\mathcal{D}(W^{\ell})$. Note that $f$ is not surjective 
because in \cite[Prop.~8.1]{KRvP} we have found a differential operator that does not commute with 
some of the other. However, the algebra $U(\lau^{\C})^{\lak^{\C}}$ is commutative by Lemma 
\ref{lem: grp alg}.

By means of Lemma \ref{lem: grp alg} we know that $U(\lau^{\C})^{\lak^{\C}}$ has 
$\Omega_{1}=\Omega_{\lak}\otimes1$ and $\Omega_{2}=1\otimes\Omega_{\lak}$ among its generators, 
where $\Omega_{\lak}\in\laZ$ is the Casimir operator. In the following subsection we calculate 
$f(\Omega_{1}+\Omega_{2})$ and $f(\Omega_{1}-\Omega_{2})$ explicitly. Upon transposing and taking 
suitable linear combinations we find the differential operators $\tilde{D}$ and $\tilde{E}$ from 
Theorem \ref{thm:differentialoperatorsP}.


\subsection{Calculation of the Casimir operators}\label{ssec:calculationCasimiropertors}

The goal of this subsection is to calculate $f(\Omega)$ and $f(\Omega')$ where $f$ is the map 
described in Definition \ref{def:to dif op pol} and where $\Omega=\Omega_{1}+\Omega_{2}$ and 
$\Omega'=\Omega_{1}-\Omega_{2}$. We proceed in a series of six steps. (1) First we provide 
expressions for the Casimir operators $\Omega$ and $\Omega'$ which (2) we rewrite according to the 
infinitesimal Cartan decomposition defined by Casselman and Mili\v{c}i\'c \cite[\S2]{CM1982}. These 
calculations are similar to those in \cite[Prop.~9.1.2.11]{WarnerII}. (3) From this expression we 
can easily calculate the $T^{\ell}$-radial parts, see Theorem \ref{thm: radial part}. The radial 
parts are differential operators for $\mathrm{End}_{M}(H^{\ell})$-valued functions on $A$. At this 
point we  see that we can extend matters to the complexification $A^\C$ of $A$ as in \cite[Ex.~3.7]{CM1982}. 
(4) We identify $\mathrm{End}_{M}(H^{\ell})\cong\C^{2\ell+1}$ and rewrite the radial 
parts of step 3 accordingly. (5) We conjugate these differential operators with $\Phi_{0}$ and (6) 
we make a change of variables to obtain two matrix-valued differential operators $f(\Omega)$ and 
$f(\Omega')$. Along the way we keep track of the differential equations for the spherical 
functions. Finally we give expressions for the eigenvalues $\Lambda_{d}$ and $\Gamma_{d}$ of 
$f(\Omega)$ and $f(\Omega')$ such that the full spherical polynomials $Q_{d}$ are the corresponding 
eigenfunctions. Following Casselman and Mili\v{c}i\'c \cite[\S2]{CM1982} the roots are considered 
as characters, hence written multiplicatively.

\textbf{(1).} First we concentrate on one factor $K\cong\SU(2)$, with Lie algebra $\lak$ and 
standard Cartan subalgebra $\lat$. The complexifications are denoted by $\lak^\C$, $\lat^\C$ and we 
use the standard basis 
\begin{equation*}
H=\begin{pmatrix}1&0\\0&-1\end{pmatrix},\quad E_{\al}=\frac{1}{2}\begin{pmatrix}0&1\\0&0\end{pmatrix},\quad E_{\al^{-1}}=\frac{1}{2}\begin{pmatrix}0&0\\1&0\end{pmatrix}
\end{equation*}
for $\lak^\C$. 
The Casimir of $K$ is given by $\Om_{\lak}=
\frac{1}{2}H^{2}+4\left\{E_{\alpha}E_{\alpha^{-1}}+E_{\alpha^{-1}}E_{\alpha}\right\}$. It is well-known 
that the matrix-elements of the irreducible unitary representation $T^\ell$ of $\SU(2)$ are 
eigenfunctions of the Casimir operator $\Om_{\lak}$ for the eigenvalue $\frac12(\ell^2+\ell)$, see 
e.g. \cite[Thm.~5.28]{KnappLGBI}. 
The roots of the pair $(\lak^\C\oplus\lak^\C,\lat^\C\oplus\lat^\C)$ are given by $R=\{(\al,1),
(\al^{-1},1),(1,\al),(1,\al^{-1})\}$. The positive roots are choosen as $R^+ = \{ (\al,1), (1,
\al^{-1})\}$, so that the two positive roots restrict to the same root $R(\lau^\C,\laa^\C)$ which 
we declare positive. The corresponding root vectors are $E_{(\al,1)}=(E_{\al},0)$, etc. Define
\begin{equation*}
E=(E_{\alpha},0)(E_{\alpha^{-1}},0)+(E_{\alpha^{-1}},0)(E_{\alpha},0).
\end{equation*}
Then we have $\Om_1= \frac12(H,0)^2+4E$ and $\Om_2 =\theta(\Om_1)$. In particular, the spherical 
function
$\Phi^\ell_{\ell_1,\ell_2}$ is an eigenfunction of $\Om_i$ for the eigenvalue $\frac12(\ell_i^2+\ell_i)$ for $i=1,2$. 

We have $(H,0)=\frac{1}{2}((H,-H)+(H,H))$ and $(0,H)=\frac{1}{2}((H,H)-(H,-H))$ and from this we find in $U(\lau^\C)$
\begin{equation}
\begin{split}
\Om\, &= \, \Om_{1}+\Om_{2}\, =\, \frac14 (H,H)^2 +\frac14 (H,-H)^2+4(E+\theta(E)),\\
\Om'\, &=\, \Om_{1}-\Om_{2}\, =\, \frac12 (H,-H)(H,H)+4(E-\theta(E)).
\end{split}
\end{equation}

\textbf{(2).} Following Casselman and Mili\v{c}i\'c \cite[\S2]{CM1982} we can express $\Omega$ and 
$\Omega'$ according to the infinitesimal Cartan decomposition of $U(\lau^{\C})$. Let $\be\in R$ 
and denote $X_{\be}=E_{\be}+\theta(E_{\be})\in\lak^{\C}$. Denote $Y^{a}=\Ad(a^{-1})Y$ for 
$a\in A$. In \cite[Lemma 2.2]{CM1982} it is proved that the equality
\[(1-\be(a)^{2})X_{\be}=\be(a)(E_{\be}^{a}-\be(a)E_{\be})\]
holds for all $a\in A_{reg}$. This is the key identity in a straightforward but tedious calculation 
to prove the following proposition which we leave to the reader.

\begin{prop} Let $a\in A_{reg}$ and $\be\in R^{+}$. Then
\begin{multline}\label{eq:explexprOmega}
 \Om =\frac{1}{16}\left((H,-H)^{2}+(H,H)^{2}\right)-\frac{2}{(\be(a)^{-1}-\be(a))^{2}}\{X_{\be}^{a}X_{\be^{-1}}^{a}+\\
X_{\be^{-1}}^{a}X_{\be}^{a}+X_{\be}X_{\be^{-1}}+X_{\be^{-1}}X_{\be}-(\be(a)+\be(a)^{-1})(X_{\be}^{a}X_{\be^{-1}}+X_{\be^{-1}}^{a}X_{\be})\}\\
+\frac{1}{4}\frac{\be(a)+\be(a)^{-1}}{\be(a)-\be(a)^{-1}}(H,-H).
\end{multline}
and 
\begin{multline}\label{eq:explexprOmegaaccent}
 \Om' =\frac{1}{8}(H,H)(H,-H)+
\frac{1}{4}\frac{\be(a)+\be(a)^{-1}}{\be(a)-\be(a)^{-1}}(H,H)\\ 
+\frac{2}{\be(a)-\be(a)^{-1}}(X_{\be^{-1}}^{a}X_{\be}-X_{\be}^{a}X_{\be^{-1}}), 
\end{multline}
\end{prop}
The calculation of \eqref{eq:explexprOmega} is completely analogous to that of 
\cite[Prop.~9.1.2.11]{WarnerII} and it is clear that \eqref{eq:explexprOmega} is invariant for 
interchanging $\be$ and $\be^{-1}$. The expression in \eqref{eq:explexprOmegaaccent} is also 
invariant for interchanging $\be$ and $\be^{-1}$ albeit that it is less clear in this case. In 
either case the expressions \eqref{eq:explexprOmega} and \eqref{eq:explexprOmegaaccent} do not 
depend on the choice of $\be\in R^{+}$.

\textbf{(3).} Following Casselman and Mili\v{c}i\'c \cite[\S 3]{CM1982} we calculate the 
$T^{\ell}$-radial parts of $\Om$ and $\Om'$. This is a matter of applying the map $\Pi_{\ell}$ from Theorem \ref{thm: radial part} to the expressions \eqref{eq:explexprOmega} and 
\eqref{eq:explexprOmegaaccent}. At the same time we note that the coefficients in 
\eqref{eq:explexprOmega} and \eqref{eq:explexprOmegaaccent} are analytic functions on $A_{reg}$. 
They extend to meromorphic functions on the complexification $A^{\C}$ of $A$ which we identify with 
$\C^{\times}$ using the map
\[a:\C^{\times}\to A^{\C}:w\mapsto a(w)=\left(\left(\begin{array}{cc}w&0\\0&w^{-1}\end{array}\right),\left(\begin{array}{cc}w^{-1}&0\\0&w\end{array}\right)\right).\]
Under this isomorphism the differential operator $(H,-H)$ translates to $w\frac{d}{dw}$. To see 
this let $g:A^{\C}\to\C$ be holomorphic and consider $(H,-H)g(a(w))$ which is equal to
\[(H,-H)g(a(w))=\left\{\frac{d}{dt}g(a(e^{t}w))\right\}_{t=0}=w\frac{d}{dw}(g\circ a)(w).\]
Following \cite{CM1982}, \cite{WarnerII} we find the following expressions for the $T^{\ell}$-radial parts of $\Omega$ and $\Omega'$;

\begin{multline}\label{eq: rad second} 
\Pi_{\ell}(\Om)=\frac{1}{16}\left(w\frac{d}{dw}\right)^{2}+\frac{1}{4}\frac{w^{2}+w^{-2}}{w^{2}-w^{-2}}w\frac{d}{dw}+\frac{1}{16}T^{\ell}(H)^{2}+\\
-\frac{2}{(w^{2}-w^{-2})^{2}}\left\{T^{\ell}(E_{\alpha})T^{\ell}(E_{\alpha^{-1}})\bullet+T^{\ell}(E_{\alpha^{-1}})T^{\ell}(E_{\alpha})\bullet+
\right.\\\left.
\bullet T^{\ell}(E_{\alpha})T^{\ell}(E_{\alpha^{-1}})+\bullet
T^{\ell}(E_{\alpha^{-1}})T^{\ell}(E_{\alpha})\right\}+\\
2\frac{w^{2}+w^{-2}}{(w^{2}-w^{-2})^{2}}\left\{T^{\ell}(E_{\alpha})\bullet
T^{\ell}(E_{\alpha^{-1}})+ T^{\ell}(E_{\alpha^{-1}})\bullet
T^{\ell}(E_{\alpha})\right\},
\end{multline}
and
\begin{multline}\label{eq: rad first}
 \Pi_{\ell}(\Om')=\frac{1}{8}T^{\ell}(H)w\frac{d}{dw}+\frac{1}{4}\frac{w^{2}+w^{-2}}{w^{2}-w^{-2}}T^{\ell}(H)+\\
\frac{2}{w^{2}-w^{-2}}\left\{T^{\ell}(E_{\alpha^{-1}})\bullet
T^{\ell}(E_{\alpha})+ T^{\ell}(E_{\alpha})\bullet
T^{\ell}(E_{\alpha^{-1}})\right\}
\end{multline}
where the bullet $(\bullet)$ indicates where to put the restricted
spherical function. The matrices $T^{\ell}(E_{\alpha})$ and $T^{\ell}(H)$ are easily calculated in 
the basis of weight vectors. Note that $T^{\ell}(E_{\alpha^{-1}})=JT^{\ell}(E_{\alpha})J$. We give 
the entries of $T^{\ell}(E_{\alpha})$ in the proof of Lemma \ref{lemma: conj wth UP}.

The following proposition is a direct consequence of Theorem \ref{thm: radial part} and Proposition 
\ref{prop: Csf eigenv}.
\begin{prop} The restricted spherical functions are eigenfunctions of the radial parts of $\Omega$ and $\Omega'$,
\begin{gather*}
\Pi_\ell(\Om)(\Phi_\ell^{\ell_1,\ell_2}|_{A^{\C}})\, = \, \frac12(\ell_1^2+\ell_1+\ell_2^2+\ell_2)\Phi_\ell^{\ell_1,\ell_2}|_{A^{\C}}, \\
\Pi_\ell(\Om')(\Phi_\ell^{\ell_1,\ell_2}|_{A^{\C}})\, = \, \frac12(\ell_1^2+\ell_1-\ell_2^2-\ell_2)\Phi_\ell^{\ell_1,\ell_2}|_{A^{\C}}.
\end{gather*}
\end{prop}

\textbf{(4).} The spherical functions $\Phi_{\ell_{1},\ell_{2}}^{\ell}$ restricted to the torus $A^{\C}$ take their values in $\mathrm{End}_{M}(H^{\ell})$ and this is a $2\ell+1$-dimensional vector space. We identify
\[\mathrm{End}_{M}(H^{\ell})\to\C^{2\ell+1}:Y\mapsto Y^{\mathrm{up}}\]
to obtain functions $(\Phi_{\ell_{1},\ell_{2}}^{\ell}|_{A^{\C}})^{\mathrm{up}}$. The reason for 
putting the diagonals up is that we want to write the differential operators as differential 
operators with coefficients in the function algebra on $A$ with values in $\mathrm{End}
(\C^{2\ell+1})$ instead of the way $\Pi_{\ell}(\Omega)$ and $\Pi_{\ell}(\Omega')$ are defined. The 
differential operators that are conjugated to act on $\C^{2\ell+1}$-valued functions are also 
denoted by $(\cdot)^{\mathrm{up}}$. The differential operators \eqref{eq: rad second} and 
\eqref{eq: rad first} that are defined for $\mathrm{End}_{M}(H^{\ell})$-valued functions conjugate 
to differential operators $\Pi_{\ell}(\Omega)^{\mathrm{up}}$ and $\Pi_{\ell}
(\Omega')^{\mathrm{up}}$ for $\C^{2\ell+1}$-valued functions. All the terms except for the last 
ones in \eqref{eq: rad second} and \eqref{eq: rad first} transform straightforwardly. 

\begin{lem}\label{lemma: conj wth UP}
The linear isomorphism $\mathrm{End}_{M}(H^{\ell})\to\C^{2\ell+1}:D\mapsto D^{\mathrm{up}}$ 
conjugates the linear map $\mathrm{End}_{M}(H^{\ell})\to\mathrm{End}_{M}(H^{\ell}):D\mapsto T^{\ell}(E_{\alpha})DT^{\ell}(E_{\alpha^{-1}})$ to 
$\C^{2\ell+1}\to\C^{2\ell+1}:D^{\mathrm{up}}\mapsto C^{\ell}D^{\mathrm{up}}$, where 
$C^{\ell}\in\mathrm{End}(\C^{2\ell+1})$ is the matrix given by
\[C^{\ell}_{p,j}=\frac{1}{4}(\ell+j)(\ell-j+1)\delta_{j-p,1}, \quad \ell\le p,j\le\ell.\]
Likewise, $D\mapsto T^{\ell}(E_{\alpha^{-1}})DT^{\ell}(E_{\alpha})$ transforms to 
$D^{\mathrm{up}}\mapsto JC^{\ell}JD^{\mathrm{up}}$, where $J$ is the anti-diagonal defined by 
$J_{ij}=\delta_{i,-j}$ with $-\ell\le i,j\le\ell$.
\end{lem}

\begin{proof} Working with the normalized weight-basis as in \cite[\S1]{Koornwinder85} we see that $T^{\ell}(E_{\alpha})$ is the matrix given by
\[
T^{\ell}(E_{\alpha})_{ij}=\delta_{i,i+1}\frac{\ell+i+1}{2}\sqrt{\frac{(\ell-i-2)!(\ell+i+2)!}{(\ell-i-1)!(\ell+i+1)!}}
\]
and $T^{\ell}(E_{\alpha^{-1}})=JT^{\ell}(E_{\alpha})J$. The lemma follows from elementary manipulations. 
\end{proof}

We collect the expressions for the conjugation of the differential operators \eqref{eq: rad second} 
and \eqref{eq: rad first} by the linear map $Y\mapsto Y^{\mathrm{up}}$ where we have used Lemma \ref{lemma: conj wth UP}.

\begin{multline}\label{eq: rad second UP}
 \Pi_{\ell}(\Omega_{1}+\Omega_{2})^{\mathrm{up}}=\frac{1}{16}\left(w\frac{d}{dw}\right)^{2}+\frac{1}{4}\frac{w^{2}+w^{-2}}{w^{2}-w^{-2}}w\frac{d}{dw}+\frac{1}{16}T^{\ell}(H)^{2}+\\
-\frac{4}{(w^{2}-w^{-2})^{2}}\left\{T^{\ell}(E_{\alpha})T^{\ell}(E_{\alpha^{-1}})+T^{\ell}(E_{\alpha^{-1}})T^{\ell}(E_{\alpha})\right\}+\\
2\frac{w^{2}+w^{-2}}{(w^{2}-w^{-2})^{2}}\left\{JC^{\ell}J+C^{\ell}\right\},
\end{multline}

\begin{multline}\label{eq: rad first UP}
 \Pi_{\ell}(\Omega_{1}-\Omega_{2})^{\mathrm{up}}=\frac{1}{8}T^{\ell}(H)w\frac{d}{dw}+\frac{1}{4}\frac{w^{2}+w^{-2}}{w^{2}-w^{-2}}T^{\ell}(H)\\ +\frac{2}{w^{2}-w^{-2}}\left\{JC^{\ell}J-C^{\ell}\right\}.
\end{multline}

The differential operators \eqref{eq: rad second UP} and \eqref{eq: rad first UP} also act on the 
full spherical functions $\Phi_{d}^{\ell,t}$. Collecting the eigenvalues of the columns in 
$\Phi_{d}^{\ell,t}$ in diagonal matrices we obtain the following differential equations:
\begin{eqnarray}
\Pi_{\ell}(\Omega_{1}+\Omega_{2})^{\mathrm{up}}\Phi_{d}&=&\Phi_{d}\Lambda_{d},\\
\Pi_{\ell}(\Omega_{1}-\Omega_{2})^{\mathrm{up}}\Phi_{d}&=&\Phi_{d}\Gamma_{d},
\end{eqnarray}
where $(\Lambda_{d})_{pj}=\frac{1}{4}\delta_{p,j}(d^2+j^2+2d(\ell+1)+\ell(\ell+2))$ and $(\Gamma_{d})_{pj}=\frac{1}{2}\delta_{p,j}j(\ell+d+1)$.
For further reference we write
\begin{eqnarray}
\Pi_{\ell}(\Omega_{1}+\Omega_{2})^{\mathrm{up}}&=&a_{2}(w)\frac{d^{2}}{dw^{2}}+a_{1}(w)\frac{d}{dw}+a_{0}(w),\label{eq: conjugation 2nd first step}\\
\Pi_{\ell}(\Omega_{1}-\Omega_{2})^{\mathrm{up}}&=&b_{1}(w)\frac{d}{dw}+b_{0}(w).\label{eq: conjugation 1st first step}
\end{eqnarray}

\textbf{(5).} Recall from Definition \ref{def: full sf} that the full spherical polynomials 
$Q_{d}^{\ell,t}$ are obtained from the full spherical functions $\Phi_{d}^{\ell,t}$ by the 
description $Q_{d}^{\ell,t}=(\Phi_{0}^{\ell,t})^{-1}\Phi_{d}^{\ell,t}$. We conjugate the 
differential operators \eqref{eq: rad second UP} and \eqref{eq: rad first UP} with $\Phi_{0}$ to 
obtain differential operators to which the polynomials $Q_{d}$ are eigenfunctions. We need a 
technical lemma.    

\begin{lem}\label{lemma:dphinul}
Let $\sigma^{\ell}:\C^{\times}\to\mathrm{End}(\C^{2\ell+1})$ be defined by $\sigma^{\ell}(w)=\ell(w^{2}+w^{-2})I+S^{\ell}$ where $S^{\ell}$ is defined by $\left(S^{\ell}\right)_{p,j}=-(\ell-j)\delta_{p-j,1}-(\ell+j)\delta_{j-p,1}$. Then
\begin{eqnarray}\label{eq: tech lem 1}
\frac{1}{2}w(w^{2}-w^{-2})\frac{d}{dw}\Phi^{\ell,t}_{0}(w)=\Phi^{\ell,t}_{0}(w)\sigma^{\ell}(w).
\end{eqnarray}
Let $\upsilon^{\ell}:\C^{\times}\to\mathrm{End}(\C^{2\ell+1})$ be defined by 
$\upsilon^{\ell}(w)=\frac{1}{8}\frac{w^{3}}{w^{4}-1}\left(\frac{1+w^{4}}
{w^{2}}U^{\ell}_{\mathrm{diag}}+U^{\ell}_{lu}\right)$, where 
$\left(U^{\ell}_{lu}\right)_{i,j}=(-2\ell+2j)\delta_{i,j+1}+(2\ell+2j)\delta_{i+1,j}$ and 
$\left(U^{\ell}_{\mathrm{diag}}\right)_{i,j}=-2i\delta_{ij}$. Then
\begin{eqnarray}\label{eq: tech lem 2}
b_{1}(w)\Phi_{0}^{\ell,t}(a(w))=\Phi_{0}^{\ell,t}(a(w))\upsilon^{\ell}(w).
\end{eqnarray}
\end{lem}

\begin{proof}
The matrix coefficients of $\Phi_{0}^{\ell,t}(a(w))$ are given by
\begin{multline}
(\Phi_{0}^{\ell,t}(a(w)))_{p,j}=\frac{(\ell-j)!(\ell+j)!(\ell-p)!(\ell+p)!}{(2\ell)!}\times\\
\sum_{r=\max(0,-p-j)}^{\min(\ell-p,\ell-j)}\frac{w^{4r-2\ell+2p+2j}}{r!(\ell-p-r)!(\ell-j-r)!(p+j+r)!},
\end{multline}
see \cite[Prop.~3.2]{Koornwinder85}. 
The matrix-valued function $b_{1}(w)$ is equal to the constant matrix
$\frac{1}{8}T^{\ell}$ where $T^{\ell}(H)_{ij}=2\delta_{ij}j$.
We can now express the matrix coefficients of the matrices in \eqref{eq: tech lem 1} and 
\eqref{eq: tech lem 2} in Laurent polynomials in the variable $w$ and comparing coefficients of 
these polynomials shows that the equalities hold.
\end{proof}

\begin{defin}
Define $\Omega_{\ell}=(\Phi_{0}^{\ell,t})^{-1}\circ\Pi_{\ell}
(\Omega)^{\mathrm{up}}\circ\Phi_{0}^{\ell,t}$ and 
$\Delta_{\ell}=(\Phi_{0}^{\ell,t})^{-1}\circ\Pi_{\ell}
(\Omega')^{\mathrm{up}}\circ\Phi_{0}^{\ell,t}$.
\end{defin}

\begin{thm}
The differential operators $\Omega_{\ell}$ and $\Delta_{\ell}$ are given by
\begin{gather} 
\Omega_{\ell}=\frac{1}{16}\left(w\frac{d}{dw}\right)^{2}+\frac{1}{4}\left\{(\ell+1)(w^{2}+w^{-2})+S^{\ell}\right\}\frac{w}{w^{2}-w^{-2}}\frac{d}{dw}+\Lambda_{0}\label{eq: 2nd up w}\\
\Delta_{\ell}=\upsilon^{\ell}(w)\frac{d}{dw}+\Gamma_{0}\label{eq: 1st up w}
\end{gather}
\end{thm}

\begin{proof}
This is a straightforward calculation using the expressions \eqref{eq: conjugation 2nd first step} 
and \eqref{eq: conjugation 1st first step}, bearing in mind that the coefficients are 
matrix-valued. In both calculations the difficult parts are taken care of by Lemma \ref{lemma:dphinul}.
\end{proof}

\textbf{(6).} The elementary zonal spherical function $\Phi^{\frac{1}{2},\frac{1}{2}}_{0}$ is 
denoted by $\phi$ and we have $\phi(a(w))=\frac{1}{2}(w^{2}+w^{-2})$. In this final step we note 
that the differential operators $\Omega_{\ell}$ and $\Delta_{\ell}$ are invariant under the maps 
$w\mapsto -w$ and $w\mapsto w^{-1}$. This shows that the differential operators can be pushed 
forward by $\phi\circ a$ to obtain differential operators on $\C$ in a coordinate $z=\phi(a(w))$. 
Using the identities $w\frac{d}{dw}(h\circ\phi)(w)=(w^{2}-w^{-2})h'(\phi(w))$, $(w\frac{d}{dw})^{2}
(h\circ\phi)(w)=(w^{2}-w^{-2})^{2}h''(\phi(a(w)))+2(w^{2}+w^{-2})h'(\phi(a(w)))$ and $(w^{2}-
w^{-2})^{2}=4(\phi(a(w))^{2}-1)$ we transform 
\eqref{eq: 2nd up w} and \eqref{eq: 1st up w} into
\begin{gather}
\widetilde{\Omega_{\ell}}=\frac{1}{4}(z^{2}-1)\left(\frac{d}{dz}\right)^{2}+\frac{1}{4}\left\{(2\ell+3)z+S^{\ell}\right\}\frac{d}{dz}+\Lambda_{0}\label{eq: 2nd up z}\\
\widetilde{\Delta_{\ell}}=\frac{1}{8}\left(2zU^{\ell}_{\mathrm{diag}}+U^{\ell}_{ul}\right)\frac{d}{dz}+\Gamma_{0}\label{eq: 1st up z}
\end{gather}

Recall that the $\mathrm{End}(\C^{2\ell+1})$-valued polynomials $R_{d}^{\ell,t}$ are defined by 
pushing forward the $\mathrm{End}(\C^{2\ell+1})$-valued functions $Q_{d}^{\ell,t}$ over $\phi\circ 
a$, see Definition \ref{def: sf pol in x}.

\begin{thm}\label{thm:DOfromgrouptoDE} The members of the family $\{R_{d}^{\ell,t}\}_{d\ge0}$ 
of $\mathrm{End}(\C^{2\ell+1})$-valued polynomials of degree $d$ are eigenfunctions of the 
differential operators $\widetilde{\Omega_{\ell}}$ and $\widetilde{\Delta_{\ell}}$ with eigenvalues 
$\Lambda_{d}$ and $\Gamma_{d}$ respecively. The transposed differential operators 
$(\widetilde{\Omega_{\ell}})^{t}$ and $(\widetilde{\Delta_{\ell}})^{t}$ satisfy
\begin{gather*}
-4(\widetilde{\Omega_{\ell}})^{t}+2(\ell^{2}+\ell)=\tilde{D},\\
-\frac{2}{\ell}(\widetilde{\Delta_{\ell}})^{t}-(\ell+1)=\tilde{E},
\end{gather*}
where $\tilde{D}$ and $\tilde{E}$ are defined in Theorem \ref{thm:differentialoperatorsP}.
\end{thm}

\begin{proof}
The only things that need proofs are the equalities of the differential operators. These follow 
easily upon comparing coefficients where one has to bear in mind the different labeling of the 
matrices involved in the two cases.
\end{proof}
Note that the differential operators $\tilde{D}$ and $\widetilde{\Omega_{\ell}}$ are invariant 
under conjugation by the matrix $J$, where $J_{i,j}=\delta_{i,-j}$. The differential operator 
$\widetilde{\Delta_{\ell}}$ is anti-invariant for this conjugation. The differential operator 
$\tilde{E}$ does not have this nice property.

\appendix

\section{Proof of Theorem \ref{thm:LDUdecompW}}\label{app:proofthmLDUdecompW}

The purpose of this appendix is to prove the LDU-decomposition of Theorem \ref{thm:LDUdecompW}. We 
prove instead the equivalent Proposition \ref{prop:thmLDUdecompW}, and we start with proving Lemma 
\ref{lem:Racah-thmLDUdecompW}. 

Note that the integral in Lemma \ref{lem:Racah-thmLDUdecompW} is zero 
by \eqref{eq:orthorelGegenbauerpols} in case $t>m$, since 
$C^{(k+1)}_{m-k}(x) U_{n+m-2t}(x)$ is a polynomial of degree $n+2m-k-2t<n-k$.

We start by proving Lemma \ref{lem:Racah-thmLDUdecompW} in the remaining case for which we use the 
following well-known formulas for connection and linearisation formulas of Gegenbauer polynomials, 
see e.g. \cite[Thm.~6.8.2]{AndrAR}, \cite[Thm.~9.2.1]{Isma};
\begin{equation}\label{eq:formulasGegenbauer}
\begin{split}
&C^{(\ga)}_n(x) = \sum_{k=0}^{\lfloor n/2\rfloor} 
\frac{(\ga-\be)_k (\ga)_{n-k}} {k!\, (\be+1)_{n-k}} 
\left( \frac{\be+n-2k}{\be}\right) C^{(\be)}_{n-2k}(x), \\
&C^{(\al)}_n(x)C^{(\al)}_m(x) = \sum_{k=0}^{m\wedge n}
\frac{(n+m-2k+\al) (n+m-2k)! (\al)_k}{(n+m-k+\al) k!} \\
&\qquad\qquad\qquad\qquad \times
\frac{ (\al)_{n-k} (\al)_{m-k} (2\al)_{n+m-k}}
{(n-k)! (m-k)! (\al)_{n+m-k} (2\al)_{n+m-2k}} 
C^{(\al)}_{n+m-2k}(x)
\end{split}
\end{equation}

\begin{proof}[Proof of Lemma \ref{lem:Racah-thmLDUdecompW}] We indicate the proof of Lemma 
\ref{lem:Racah-thmLDUdecompW}, so that the reader can easily fill in the details. Calculating the 
product of two Gegenbauer polynomials as a sum using the linearisation formula of 
\eqref{eq:formulasGegenbauer} and expanding the Chebyshev polynomial $U_{n+m-2t}(x) = 
C^{(1)}_{n+m-2t}(x)$ in terms of Chebyshev polynomials with parameter $k+1$ using the linearisation 
formula of \eqref{eq:formulasGegenbauer}, we can rewrite the integral as a double sum with an 
integral of Chebyshev polynomials that can be evaluated using the orthogonality relations 
\eqref{eq:orthorelGegenbauerpols} reducing the integral of Lemma \ref{lem:Racah-thmLDUdecompW} to 
the single sum 
\begin{gather*}
\sum_{r= \max(0,t-k)}^{\min(t,m-k)}  
\frac{(m+n-k+1-2r)}{(m+n-k+1-r)} 
\frac{(k+1)_r (k+1)_{n-k-r} (k+1)_{m-k-r}  }
{r!\, (m-k-r)!\, (n-k-r)!\, } \\ \times
\frac{(2k+2)_{m+n-2k-r} (-k)_{k+r-t} (n+m-t-k-r)!}{(k+1)_{m+n-2k-r}\, (k-t+r)!\, (k+2)_{n+m-t-k-r}} 
\frac{\sqrt{\pi}\,  \Ga(k+\frac32)}{(k+1)\, \Ga(k+1)}.
\end{gather*}
Assuming for the moment that $k\geq t$, so the sum is $\sum_{r=0}^{\min(t,m-k)}$. Then this sum can 
be written as a very-well-poised ${}_7F_6$-series
\begin{gather*}
\frac{\sqrt{\pi}\, \Ga(k+\frac32)}{(k+1)\, \Ga(k+1)}
\frac{(k+1)_{m-k}}{(m-k)!} \frac{(k+1)_{n-k}}{(n-k)!}  
\frac{(2k+2)_{m+n-2k}}{(k+1)_{m+n-2k}} \frac{(-k)_{k-t}}{(k-t)!} 
\frac{(n+m-t-k)!}{(k+2)_{n+m-t-k}} \\
\rFs{7}{6}{\frac12(k-m-n+1), k+1, k-m, k-n,k-m-n-1, -t, t-m-n-1}
{\frac12(k-m-n-1), -m, -n, -m-n-1, k-t+1, -n-m+k+t}{1}
\end{gather*}
Using Whipple's transformation \cite[Thm.~3.4.4]{AndrAR}, \cite[\S 4.3]{Bail} of a very-well-poised 
${}_7F_6$-series to a balanced ${}_4F_3$-series, we find that the ${}_7F_6$-series can be written 
as 
\[
\frac{(k-m-n)_t\, (-t)_t}{(k-t+1)_t\, (-m-n-1)_t}
\rFs{4}{3}{-k, k+1, -t, t-m-n-1}{-n,\, -m,\, 1}{1}.
\]
Simplifying the shifted factorials and recalling the definition of the Racah polynomials 
\eqref{eq:defRacahpols} in terms of a balanced ${}_4F_3$-series gives the result in case $k\geq t$. 

In case $k\leq t$ we have to relabel the sum, which turns out again to be a very-well-poised 
${}_7F_6$-series which can be transformed to a balanced ${}_4F_3$-series. The resulting balanced 
${}_4F_3$-series is not a Racah polynomial as in the statement of Lemma 
\ref{lem:Racah-thmLDUdecompW}, 
but it can be transformed to a Racah polynomial using Whipple's transformation for 
balanced ${}_4F_3$-series \cite[Thm.~3.3.3]{AndrAR}. Keeping track of the constants proves Lemma 
\ref{lem:Racah-thmLDUdecompW} in this case. 
\end{proof}

As remarked in Section \ref{sec:LDU-weight}, Theorem \ref{thm:LDUdecompW} follows from Proposition 
\ref{prop:thmLDUdecompW}. In order to prove Proposition \ref{prop:thmLDUdecompW} we assume 
$\al_t(m,n)$ to be known \eqref{eq:defmatrix_W} and to find $\be_k(m,n)$.  Given the explicit 
$\al_t(m,n)$, and  
multiplying by $\sqrt{1-x^2}\, U_{n+m-2t}(x)$ and integrating 
we find from Lemma \ref{lem:Racah-thmLDUdecompW} 
\begin{equation}\label{eq:pfThmLDUW-result1}
\al_t(m,n) \frac{\pi}{2} = 
 \sum_{k=0}^m \be_k(m,n) C_k(m,n) R_k(\la(t);0,0,-m-1,-n-1)
\end{equation}
where 
\[
C_k(m,n) = \frac{\sqrt{\pi}\, \Ga(k+\frac32)}{(k+1)}
\frac{(k+1)_{m-k}}{(m-k)!} \frac{(k+1)_{n-k}}{(n-k)!}  
\frac{(-1)^{k}\, (2k+2)_{m+n-2k}\, (k+1)!}{(n+m+1)!}.
\]
Using the orthogonality relations for the Racah polynomials, see \cite[p.~344]{AndrAR}, \cite[\S 1.2]{KoekS}, 
\begin{equation*}
\begin{split}
&\sum_{t=0}^m (m+n+1-2t)\, R_k(\la(t);0,0,-m-1,-n-1) \, R_l(\la(t);0,0,-m-1,-n-1)  \\
&\qquad\qquad= \de_{k,l} \frac{(n+1)(m+1)}{(2k+1)} \frac{(m+2)_k (n+2)_k}{(-m)_k (-n)_k}
\end{split}
\end{equation*}
we find the following explicit expression for $\be_k(m,n)$ 
\begin{equation}\label{eq:pfThmLDUW-result2}
\begin{split}
\be_k(m,n) = & \frac{1}{C_k(m,n)} \frac{(2k+1)}{(n+1)(m+1)} \frac{(-m)_k (-n)_k}{(m+2)_k (n+2)_k}\\ & \times \sum_{t=0}^m (m+n+1-2t) R_k(\la(t);0,0,-m-1,-n-1) \al_t(m,n) \frac{\pi}{2} 
\end{split} 
\end{equation} 
Now Proposition \ref{prop:thmLDUdecompW}, and hence Theorem \ref{thm:LDUdecompW}, follows from the 
following summation and simplifying the result. 

\begin{lem}\label{lem:sumRacah} For $\ell\in\frac12\N$,  $n,m,k\in\N$ with 
$0\leq k\leq m\leq n$ we have
\begin{gather*}
 \sum_{t=0}^m (-1)^{t} \frac{(n-2\ell)_{m-t}}{(n+2)_{m-t}}
\frac{(2\ell+2-t)_t}{t!}
 (m+n+1-2t) R_k(\la(t);0,0,-m-1,-n-1) \\
 =  
(-1)^{m+k} \frac{(2\ell+k+1)!\ (2\ell-k)!}{(2\ell+1)!} 
\frac{(n+1)}{m!\, (2\ell -m)!}
\end{gather*}
\end{lem}

\begin{proof} Start with the left hand side and insert the ${}_4F_3$-series for the 
Racah polynomial and interchange summations to find
\begin{gather*}
\sum_{j=0}^k \frac{(-k)_j\, (k+1)_j}{j!\, j!\, (-m)_j \, (-n)_j} 
\frac{(n-2\ell)_{m}}{(n+2)_{m}}  \sum_{t=j}^m  (-1)^{t} \frac{(-1-n-m)_{t}}{(2\ell-n-m+1)_{t}}\\
\times 
\frac{(-2\ell-1)_t}{t!} (-1)^t
 (m+n+1-2t) (-t)_j (t-m-n-1)_j 
\end{gather*}
Relabeling the inner sum  $t=j+p$ shows that the inner sum equals
\begin{multline*}
 (-1)^j \frac{(-1-n-m)_{2j}}{(2\ell-n-m+1)_{j}}
(-2\ell-1)_{j} (1+m+n-2j)  \\ \times 
\sum_{p=0}^{m-j}  \frac{(-1-n-m+j)_{p}}{(2\ell-n-m+1+j)_{p}}
\frac{(-2\ell-1+j)_{p}}{p!} \\ 
\times\frac{(1+\frac12(-1-m-n+2j))_p}{(\frac12(-1-m-n+2j))_p}
 \frac{(-1-m-n+2j)_p}{(-1-m-n+j)_p} 
\end{multline*}
and the sum over $p$ is a hypergeometric sum. Multiplying by $\frac{(j-m)_p (j-n)_p}{(j-m)_p (j-n)_p}$
the sum can be written as a very-well-poised ${}_5F_4$-series 
\begin{gather*}
\rFs{5}{4}{1+\frac12(-1-m-n+2j),\, -1-m-n+2j,\, -1-2\ell+j,\, j-m,\, j-n}
{\frac12(-1-m-n+2j), 2\ell-n-m+j+1,\, j-n, \, j-m}{1} \\
 = \frac{(-m-n+2j)_{m-j}}{(-m-n+j+1+2\ell)_{m-j}}
\frac{(-m+1+2\ell)_{m-j}}{(-m+j)_{m-j}}
\end{gather*}
by the terminating Rogers-Dougall summation formula \cite[\S 4.4]{Bail}.

Simplifying shows that the left hand side of the lemma is equal to the single sum 
\begin{multline*}
\frac{(n-2\ell)_{m}}{(n+2)_{m}} (-1)^m (n+m+1) \frac{(-n-m)_m}{(2\ell-n-m+1)_m}
\frac{(2\ell+1-m)_m}{m!} \\ \times \sum_{j=0}^k \frac{(-k)_j\, (k+1)_j}{j!\, j!\,} 
\frac{(-2\ell-1)_j}{(-2\ell)_j} 
\end{multline*}
which can be summed by the Pfaff-Saalsch\"utz summation \cite[Thm.~2.2.6]{AndrAR}, \cite[(1.4.5)]{Isma}. This proves the lemma after some simplifications.
\end{proof}

\section{Moments}\label{app:moments}

In this appendix we give an explicit sum for the generalised moments for $W$. 
By the explicit expression 
\[
U_r(x) \, = \, (r+1) \, \rFs{2}{1}{-r, \, r+2}{\frac32}{\frac{1-x}{2}}
\]
we find 
\begin{equation*}
\begin{split}
&\int_{-1}^1 (1-x)^n U_r(x) \sqrt{1-x^2}\, dx \,\\  =\, & 
(r+1) \sum_{k=0}^r \frac{(-r)_k (r+2)_k}{k!\, (\frac32)_k} 2^{-k} 
2^{n+k+2} \frac{\Ga(n+k+\frac32)\, \Ga(\frac32)}{\Ga(n+k+3)} \\
=\, & (r+1) 2^{n+2} \frac{\Ga(n+\frac32)\, \Ga(\frac32)}{\Ga(n+3)} 
\rFs{3}{2}{-r, r+2, n+\frac32}{\frac32,\, n+3}{1} 
 \\ =\, & 
= (r+1) 2^{n+2} \frac{\Ga(n+\frac32)\, \Ga(\frac32)}{\Ga(n+3)} \frac{(-n)_r}{(n+3)_r}
\end{split}
\end{equation*}
using the beta-integral in the first equality and the Pfaff-Saalsch\"utz summation \cite[Thm.~2.2.6]{AndrAR}, \cite[(1.4.5)]{Isma} in the last equality.
For $m\leq n$, the explicit expression \eqref{eq:defmatrix_W} gives the following generalised moments
\begin{equation}
\begin{split}
&\int_{-1}^1 (1-x)^p W(x)_{nm}\, dx = 
2^{p+2} \frac{\Ga(p+\frac32)\, \Ga(\frac32)}{\Ga(p+3)} 
\frac{(2\ell+1)}{n+1}\frac{(2\ell-m)!m!}{(2\ell)!} \\
& \qquad\times \sum_{t=0}^m (-1)^{m-t} \frac{(n-2\ell)_{m-t}}{(n+2)_{m-t}}
\frac{(2\ell+2-t)_t}{t!} (n+m-2t+1)
\frac{(-p)_{n+m-2t}}{(p+3)_{n+m-2t}}
\end{split}
\end{equation}

\subsection*{Acknowledgement.} We thank Juan A. Tirao for his suggestion on how to obtain the first 
order differential operator from the Casimir operators in Section 
\ref{sec:grouptheoreticderivation}. The work by Zurri\'an mentioned in the Introduction is now available in 
\cite{PachTZ}.
We also thank Erik van den Ban for pointing out to one of us 
(MvP) the paper \cite{CM1982} and for explaining the result of this paper. We also thank Michel Brion for 
his help in the formulation of the results of Section \ref{sec:grouptheoreticderivation}. 
The work of Pablo 
Rom\'an on this paper was done while employed by the Katholieke Universiteit Leuven, Belgium, 
through grant OT/08/33 of the KU Leuven and grant P06/02 of the Belgian Interuniversity Attraction 
Pole. Pablo Rom\'an thanks the KU Leuven for hospitality.


\end{document}